\documentclass[12pt]{amsart}
\usepackage[leqno]{amsmath}
\usepackage[all]{xy}
\usepackage[boxsize=1.25em, centerboxes]{ytableau}

\usepackage{appendix}

\usepackage[hidelinks]{hyperref}

\hypersetup{
    colorlinks = true, 
    linkbordercolor = {blue},
    citecolor = {blue}
}

\usepackage{amsthm,amsfonts, amssymb, amscd}
\usepackage{youngtab}
\usepackage{ytableau}
\usepackage[all]{xy}
\usepackage{multirow, multicol}
\usepackage{euscript}
\usepackage{cite}

\usepackage{tikz, tikz-cd}
\usetikzlibrary{cd, matrix,arrows, decorations.markings}
\usetikzlibrary{positioning}

\usepackage[normalem]{ulem}

\makeatletter
\newcommand{\leqnomode}{\tagsleft@true\let\veqno\@@leqno}
\newcommand{\reqnomode}{\tagsleft@false\let\veqno\@@eqno}
\makeatother

\let\oldtocsection=\tocsection
\let\oldtocsubsection=\tocsubsection
\renewcommand{\tocsection}[2]{\hspace{0em}\oldtocsection{#1}{#2}}
\renewcommand{\tocsubsection}[2]{\hspace{1em}\oldtocsubsection{#1}{#2}}

\newtheorem{thm}{Theorem}[section]
\newtheorem{lemma}[thm]{Lemma}
\newtheorem{cor}[thm]{Corollary}
\newtheorem{prop}[thm]{Proposition}

\newtheorem{conj}[thm]{Conjecture}
\theoremstyle{remark}
\newtheorem{rem}[thm]{Remark}
\theoremstyle{remark}
\theoremstyle{definition}

\newtheorem{eg_no_qed}[thm]{Example}
\newenvironment{example}[1][]{\begin{eg_no_qed}[#1]\pushQED{\qed}}{\popQED\end{eg_no_qed}}
\AfterEndEnvironment{eg}{\noindent\ignorespaces}

\newtheorem{THM}{Theorem}
\theoremstyle{definition}

\theoremstyle{definition}

\newtheorem{Question}{Question}
\theoremstyle{definition}

\newtheorem{definition}[thm]{Definition}

\numberwithin{equation}{section}
\newcommand\junk[1]{}

\newcommand{\C}{\mathbb{C}}           
\newcommand{\Z}{\mathbb{ Z}}           
 
\renewcommand{\ker}{\operatorname{ker }}
\newcommand{\diag}{\operatorname{diag}}

\newcommand{\ev}{\operatorname{ev}}

\newcommand{\rk}{\operatorname{rk}}

\newcommand{\fg}{{\mathfrak g}}

\newcommand{\cc}{\mathcal{C}}
\newcommand{\cd}{\mathcal{D}}

 \newcommand{\ci}{\mathcal{I}}
 \newcommand{\cj}{\mathcal{J}}
 \newcommand{\ck}{\mathcal{K}}
 \newcommand{\cl}{\mathcal{L}}
 \newcommand{\cm}{\mathcal{M}}
 
 \newcommand{\co}{\mathcal{O}}
 \newcommand{\cp}{\mathcal{P}}

 \newcommand{\ct}{\mathcal{T}}

 \newcommand{\cx}{\mathcal{X}}

\renewcommand{\tilde}{\widetilde}

\renewcommand{\bar}[1]{\overline{#1}}

\newcommand{\fX}{\mathcal{Y}}

\newcommand{\op}{\mathrm{op}}
\newcommand{\Spec}{\mathrm{Spec}}

\newcommand{\Br}{\mathrm{Br}}

\newcommand{\semi}{\mathsf{s}}
\newcommand{\nilp}{\mathsf{n}}

\newcommand{\mx}{\mathsf{x}}

\newcommand{\LT}{\mathrm{LT}}
\newcommand{\LM}{\mathrm{LM}}
\newcommand{\M}{\mathrm{M}}

\newcommand{\cjt}{\mathcal{J}^t} 
\newcommand{\cpt}{\mathcal{P}^t}

\begin{document}

\title[Matrix Hessenberg schemes over the minimal sheet]{Matrix Hessenberg schemes over the minimal sheet}

\author{Rebecca Goldin}
\address{Department of Mathematical Sciences\\ George Mason University\\ 4400 University Drive\\ Fairfax, VA\\ 22030\\ USA}
\email{rgoldin@gmu.edu }

\author{Martha Precup}
\address{Department of Mathematics\\ Washington University in St.~Louis \\ One Brookings Drive\\ St.~Louis, Missouri\\ 63130\\ USA  }
\email{martha.precup@wustl.edu}

\date{\today}

\begin{abstract} We study families of matrix Hessenberg schemes in the affine scheme of complex $n\times n$ matrices, each defined over a fixed sheet in the Lie algebra $\mathfrak{gl}_n(\C)$. It is well known that such families over the regular sheet are flat, and every regular Hessenberg scheme degenerates to a regular nilpotent Hessenberg scheme.  This paper explores whether flat degenerations exist outside of the regular case. 

For each matrix Hessenberg scheme, we introduce a one-parameter family of matrix Hessenberg schemes that degenerates it to a specific nilpotent Hessenberg scheme. Our main theorem states that, when the family lies over the minimal sheet in $\mathfrak{gl}_n(\C)$, this degeneration is flat. 
The proof leverages commutative algebra on the polynomial ring to identify the structure of the family concretely, and we explore several applications. We conjecture that flatness holds for these families over other sheets as well. 
\end{abstract}

\maketitle


\section{Introduction} 
Hessenberg varieties form large class of subvarieties of the flag variety that have arisen in diverse contexts across mathematical fields. 
Motivated by Hessenberg matrices and algorithms for efficiently calculating eigenvalues in numerical analysis, Hessenberg varieties in the flag variety of $GL_n(\C)$ were first introduced by De~Mari and Shayman \cite{DS1988} and later defined in all Lie types by De Mari, Procesi, and Shayman \cite{DPS1992}.  Many questions regarding their geometry and topology remain open, or are answered in only a few cases. 

Examples of Hessenberg varieties include many well-studied cases such as Springer fibers~\cite{Spaltenstein1976, Steinberg}, Peterson varieties~\cite{Kostant1996}, and permutohedral varieties~\cite{DPS1992}. Hessenberg varieties also overlap with other important subvarieties of the flag variety such as Schubert varieties~\cite{EPS2023, Tymoczko2006A, Abe-Crooks2016} and $K$ orbits~\cite{EPS-online}. Hessenberg varieties and their generalizations arise in the context of Schubert calculus\cite{Carlsson-Oblomkov2019, Drellich2015, Anderson-Tymoczko2010, Insko-Tymoczko-Woo}, representation theory\cite{GKM-affine, Oblomkov-Yun2016}, and combinatorics~\cite{Carlsson-Mellit2021, MR4116638}. The topology of regular semisimple Hessenberg varieties has been tied to the study of chromatic quasisymmetric functions via an action of the symmetric group on its cohomology~\cite{Shareshian-Wachs2016, Brosnan-Chow2018}, motivating new advances in the area; see, for example, \cite{Harada-Precup2019, Aberu-Nigro2022A, Aberu-Nigro2022B, MR4692364}.

Due to their geometric diversity, the theory of Hessenberg varieties has traditionally been advanced by a focus on special cases. This paper takes a different approach: we study families of Hessenberg varieties united only by a few essential properties of the defining parameters. One-parameter families of regular Hessenberg varieties were studied by Abe--Dedieu--Galetta--Harada~\cite{ADGH2018} and Abe--Fujita--Zeng~\cite{Abe-Fujita-Zeng2020}. Our work  generalizes that approach by organizing type A Hessenberg varieties into one-parameter families with the goal of understanding how the members of each family differ. In particular, we examine when these varieties vary continuously, that is, when the families we define are \emph{flat}.

Let $n$ be a positive integer, $G=GL_n(\C)$, and $B\subset G$ be the Borel subgroup of upper triangular matrices.  The collection of flags 
\[
V_\bullet:=(0\subset V_1\subset V_2\subset \cdots \subset V_{n-1}\subset V_n=\C^n)
\]
with $\dim V_i =i$ for all $i$ is the \emph{flag variety} of $G$.  Recall that this variety can be naturally identified with the homogeneous space $G/B$. A Hessenberg variety $Y_{\mx, h}$ is a subvariety of $G/B$ that depends on two inputs: a matrix $\mx\in \mathfrak{gl}_n(\C) = \mathrm{Lie}(G)$ and a non-decreasing function $h: \{1,\ldots, n\} \to \{1,\ldots, n\}$ such that $h(i)\geq i$, called a \emph{Hessenberg function}. Explicitly,
\begin{eqnarray}\label{eqn.hessdef}
Y_{\mx, h} := \{V_\bullet \mid \mx(V_i)\subseteq V_{h(i)} \text{ for all $i$} \}
\end{eqnarray}
is the \emph{Hessenberg variety} corresponding to the matrix $\mx$ and Hessenberg function $h$. If $\mx$ is a nilpotent (respectively, semisimple) element of $\mathfrak{gl}_n(\C)$ then we say that $Y_{\mx, h}$ is a \emph{nilpotent} (respectively, \emph{semisimple}) \emph{Hessenberg variety}.

The geometry of the Hessenberg variety $Y_{\mx, h}$ depends on the choice of matrix $\mx$ and Hessenberg function $h$. In one extreme case, when $h(i)=n$ for all $i$, the Hessenberg variety is equal to the flag variety of $G$ for all matrices $\mx$.  In another, when $h(i)=i$ for all $i$, the structure of $Y_{\mx, h}$ depends on the matrix $\mx$ (or its conjugacy class) a great deal. For example, if $\mx$ is nilpotent then $Y_{\mx, h}$ is the Springer fiber over $\mx$ and if $\mx$ is semisimple then $Y_{\mx, h}$ is a disjoint union of products of smaller dimensional flag varieties.

Let $\mx\in \mathfrak{gl}_n(\C)$ be a regular matrix, that is, a matrix such that the conjugacy class $\co_{\mx}$ of $\mx$ in $\mathfrak{gl}_n(\C)$ is maximum dimensional. When $\mx$ is regular, we say that $Y_{\mx, h}$ is a \textit{regular Hessenberg variety}. Assume further that $h(i)>i$ for all $1\leq i<n$.  We call such a Hessenberg function \textit{indecomposable}, because in this case, the regular Hessenberg variety $Y_{\mx, h}$ is not a disjoint union of products of smaller dimensional regular Hessenberg varieties. 
A regular Hessenberg variety with indecomposible Hessenberg function is reduced and irreducible, of dimension $\sum_i \left(h(i)-i\right)$~\cite{Abe-Fujita-Zeng2020, Precup2018}. Note that this dimension formula does not depend on the specific regular matrix $\mx$. The geometric reason for this ``dimension stability'' is that the family regular Hessenberg varieties is flat, as we now explain.

A scheme $\fX$ \emph{admits a flat degeneration to a scheme $\fX_0$} if there exists a flat morphism $\fX_t \to \Spec\ \!\C[t]$ of schemes such that the scheme-theoretic fiber over the closed point $t\neq 0$ is isomorphic to $\fX$ and the scheme-theoretic fiber over $0$ is $\fX_0$.  We we call $\fX$ the \emph{general fiber} and $\fX_0$ the \emph{special fiber} or \emph{degenerate fiber}. We also refer to the scheme $\fX_t$ as a \emph{flat family}.  Abe, Fujita, and Zeng proved that every regular Hessenberg variety is reduced and admits a flat degeneration to a regular nilpotent Hessenberg variety~\cite[Cor.~6.5]{Abe-Fujita-Zeng2020}. The goal of this paper is to explore whether such a degeneration exists outside of the regular case.   Outside of the regular case, one needs to disambiguate the \emph{scheme} defined by the equations describing the inclusions~\eqref{eqn.hessdef} 
from its underlying variety. To that end, we define the \emph{matrix Hessenberg scheme} as a lift $\fX_{\mx, h}$ of $Y_{\mx, h}$ in $\Spec\ \!\C[M_n]$, where $M_n$ denotes the set of $n\times n$ matrices with complex entries.

The driving question of this paper is:
\begin{Question}\label{quest1} Does every matrix Hessenberg scheme $\fX_{\mx, h}$ admit a flat degeneration to a nilpotent matrix Hessenberg scheme? 
\end{Question}

Our main result is that the answer is yes when $\mx$ in the \emph{minimal sheet} of $\mathfrak{gl}_n(\C)$. Recall that a \emph{partition of $n$} is a sequence $\lambda = (\lambda_1, \ldots ,\lambda_\ell)$ of non-decreasing positive integers whose sum is $n$.  There is a decomposition of $\mathfrak{gl}_n(\C) = \bigsqcup_{\lambda} \fg_\lambda$ into subsets called \emph{sheets}, each indexed by a partition $\lambda$ of $n$. The group $G$ acts on $\mathfrak{gl}_n(\C)$ by conjugation and each conjugacy class is assigned to a particular sheet using the Jordan canonical form; see Section~\ref{sec.sheets}. The minimal sheet is the union of conjugacy classes in $\mathfrak{gl}_n(\C)$ of minimum positive dimension and the partition indexing the minimal sheet is $(2,1, 1, \dots, 1)$, also written $(2, 1^{n-2})$. We call any Hessenberg variety $Y_{\mx, h}$ with $\mx\in \fg_{(2,1^{n-2})}$ a \emph{minimal Hessenberg variety} and similarly for the associated matrix Hessenberg scheme.  Our main result is the following; see also Theorem~\ref{thm.flat}.

\begin{THM}\label{thm.flatfamily} Let $h: \{1,\ldots, n\}\to \{1,\ldots, n\}$ be a Hessenberg function.  For all $\mx\in \fg_{(2,1^{n-2})}$ there exists a curve $\mx_t$ in the minimal sheet $\fg_{(2,1^{n-2})}$ and a flat morphism
\begin{eqnarray}\label{eqn.flatmorph}
\fX_{\mx_t, h} \to \Spec\ \!\C[t]
\end{eqnarray}
with general fiber $\fX_{\mx, h}$ and special fiber $\fX_{\nilp,h}$,
where $\nilp$ is an element of the minimal nilpotent orbit in $\mathfrak{gl}_n(\C)$. Furthermore, the flat family $\fX_{\mx_t, h}$ is reduced. 
\end{THM}

More broadly, we provide the tools to ask Question~\ref{quest1} for other sheets, by defining a curve $\mx_t$ in the sheet $\fg_\lambda$, and a corresponding family of Hessenberg schemes over $\mx_t$, for any partition $\lambda$. In Conjecture~\ref{conj}, we conjecture that this family is flat over every sheet. 
In the case of the minimal sheet, the proof of Theorem~\ref{thm.flatfamily} follows from an explicit description of the associated primes for the ideal defining $\fX_{\mx_t, h}$, stated in Theorem~\ref{thm.assprimes}.

The minimal sheet has two particularly convenient elements, mainly the nilpotent matrix $\nilp:=E_{1n}$, the matrix with unique nonzero entry equal to $1$ in position $(1,n)$, and semisimple matrix $\semi=\diag(1,0,\ldots, 0)$.  
These choices are representative in that each matrix Hessenberg scheme is isomorphic to either $\fX_{\semi, h}$ or $\fX_{\nilp, h}$ by Lemma~\ref{lem:equality}.
The general fiber $\fX_{\semi, h}$ is studied in~\cite{GP-semisimple}, where we prove that this matrix Hessenberg scheme is reduced and we identify each irreducible component explicitly. 

Though the  family $\fX_{\mx_t, h}$ and its general fiber $\fX_{\semi, h}$ are reduced, the special fiber $\fX_{\nilp, h}$ need not be.  Since $\nilp$ is $B$-invariant,  the Hessenberg variety $Y_{\nilp, h}$ is a union of Schubert varieties. Tymoczko identified the particular union when $h$ is indecomposable \cite{Tymoczko2006A}, and  Abe and Crooks extended the result to include the case that $h$ is decomposable \cite{Abe-Crooks2016}. In Proposition~\ref{prop.nilp.components}, we prove a scheme theoretic version of Tymoczko's decomposition for the indecomposable clase. We then connect it with our results from~\cite{GP-semisimple} for the semisimple Hessenberg varieties over the minimal sheet, via a flat degeneration defined on each component; see Theorem~\ref{thm.components.indecomp}. 

\begin{THM}\label{thm.bijectivedegeneration} Suppose $h:\{1,\ldots, n\}\to \{1,\ldots, n\}$ is an indecomposable Hessenberg function.  The flat degeneration described in Theorem~\ref{thm.flatfamily} yields a flat degeneration from each component of the semisimple matrix Hessenberg scheme $\fX_{\semi, h}$ to a unique component of the nilpotent matrix Hessenberg scheme~$\fX_{\nilp, h}$. In particular, there is a bijection between components of each minimal semisimple Hessenberg variety and minimal nilpotent Hessenberg variety when $h$ is indecomposable.
\end{THM}

A component-wise degeneration does not generally exist outside the indecomposable case. In Example~\ref{ex.1234}, we exhibit a (flat) degeneration in which the special fiber consists of three Schubert varieties with nontrivial multiplicity, while the general fiber consists of four reduced components. We also prove $\fX_{\nilp, h}$ is not reduced when $h$ is decomposable; see Theorem~\ref{thm:decomposable}. 

\begin{THM}\label{thm4} The minimal nilpotent matrix Hessenberg scheme $\fX_{\nilp, h}$ is reduced if and only if the Hessenberg function $h$ is indecomposable.
\end{THM}

The paper is organized as follows. Section~\ref{sec.background} reviews the relevant background on matrix Schubert varieties and matrix Hessenberg schemes. In section~\ref{sec.sheets} we review sheets of the Lie algebra $\mathfrak{gl}_n(\C)$, construct the curve $\mx_t$, and define a family of Hessenberg schemes over the curve. In Section~\ref{sec.flat}, we recall the definition of a flat family, and describe two families known to be flat thanks to a phenomenon known as ``miracle flatness"; we elaborate in detail on how this technique applies in a restricted setting in Appendix~\ref{sec.miracleflatnessproofs}. Miracle flatness does not apply in the minimal sheet -- this failure suggests the need for other methods. One enormous benefit of lifting the Hessenberg variety to the matrix space that the
geometry of the Hessenberg variety can be studied using commutative algebra in the
relatively simple setting of the polynomial ring in $n^2$ variables.
Section~\ref{sec.technical} contains the technical heart of the paper, in which we examine the defining ideal of a family of matrix Hessenberg schemes in detail using Gr\"obner bases.  In Section~\ref{sec.FlatMinSheet} we use these technical results to examine the family of Hessenberg schemes over the minimal sheet in detail, proving the main theorem. The implications of these results for the special (nilpotent) fiber can be found in Section~\ref{sec.nilpotent}.

\subsection{Acknowledgments} 
The first author was supported by National Science Foundation grant \#2152312 and the second author was supported by National Science Foundation grants \#1954001 and \#2237057. We also thank Rahul Singh and Allen Knutson for several helpful conversations.


\section{Relevant matrix schemes}\label{sec.background}

Let $\mathbf{z}$ denote the set of $n^2$ variables $\{z_{ij}\mid 1\leq i,j \leq n\}$ and $\C[\mathbf{z}] = \C[z_{11}, z_{12}, \ldots, z_{nn}]$. Throughout this paper, we identify $z_{ij}$ with the coordinate function on the scheme of $n\times n$ matrices $M_n(\C)=\Spec\ \!\C[\mathbf{z}]$ returning the $(i,j)$ entry.  Let $G=GL_n(\C)$ denote the algebraic group of $n\times n$ complex invertible matrices.  As is typical, we identify $G$ as a principal open subset in $M_n(\C)$.

As in the introduction, we denote by $B\subset G$ the Borel subgroup of upper-triangular matrices.
Let $\pi: G\to G/B$ be the natural projection.  The morphism $\pi$ defines correspondences between (right) $B$-invariant subschemes of $G$, their closures in $M_n(\C)$, and subschemes of $G/B$. This section introduces two key instances of these correspondences: matrix Schubert varieties and matrix Hessenberg schemes.

For $n$ a positive integer, let $[n]:=\{1,2,\ldots, n\}$ and $S_n$ be the symmetric group on $n$ letters.  We typically denote a permutation $w\in S_n$ using its one-line notation, i.e.  $w=[w(1), w(2), \dots, w(n)]$. 
We omit the commas when $n<10$, so that, for example, $[6,5,1,2,4,3] = [651243]$. The \emph{length} of a permutation $w\in S_n$ is $\ell(w):= |\{(i,j) \mid 1\leq i<j\leq n, w(i)>w(j)\}|$. The longest permutation in $S_n$ is denoted by $w_0:=[n,n-1,\ldots, 2, 1]$.


\subsection{Matrix Schubert varieties}\label{sec:matrixSchubs}  
The Weyl group of $G$ may be identified with $S_n$. For each permutation $w$, let $\dot w\in G$ denote the permutation matrix whose entry is $1$ in row $i$ and column $j$ if $w(j)=i$, and $0$ otherwise.  There are finitely many left $B$ or $B_-:= \dot w_0 B \dot w_0$ orbits on $G/B$, in either case indexed by $S_n$. We call the closures of $B$ orbits on $G/B$ \emph{Schubert varieties,} and the closure of $B_-$ orbits on $G/B$ \emph{opposite Schubert varieties}, denoted respectively by
$$
X_w= \overline{B\dot wB/B} \; \text{ and } \;  X^\op_w= \overline{B_- \dot wB/B}.
$$ 
These varieties are known to be reduced and irreducible.

Given $w\in S_n$, we define the associated \emph{matrix Schubert variety} by $\cx_w:= \overline{B\dot wB}$ in $M_n(\C)=\Spec\ \!\C[\mathbf{z}]$.  Similarly, we define the \emph{opposite matrix Schubert variety} by $\cx_w^{\op}:= \overline{B_-\dot wB}$.\footnote{These naming conventions are different than those of \cite{Knutson-Miller2005} and \cite{MS05}, who reserve the term {\em matrix Schubert variety} to indicate the $B_-\times B$-orbit closures.} These  satisfy the relationship
\begin{equation}\label{eq:oppositetoregularSchubert}
\dot w_0 \cx_w^{\op}  = \cx_{w_0w}.
\end{equation}
Matrix Schubert varieties are characterized by the following defining ideals. 
For each pair $p,q$ with $1\leq p,q\leq n$, let $r_{pq}(w)$ be the rank of the northwest $p\times q$ submatrix of $\dot w$. One may check that
\[
r_{pq}(w) = \left| \{w(1), \ldots, w(q)\} \cap \{1, \ldots, p\} \right|.
\]
Let $Z = (z_{ij})_{1\leq i, j \leq n}$ be a matrix with variable entries and $Z_{p\times q}$ denote the submatrix given by northwest $p\times q$ submatrix of $Z$. The \emph{Schubert determinantal ideal $\ci_w$ corresponding to the permutation $w$} is the ideal in $\C[\mathbf{z}]$ generated by minors in $Z_{p\times q}$ of size $1+r_{pq}(w)$, over all  $p,q$.\footnote{Some authors associate to $w\in S_n$ the variety $\overline{B_- \dot{w}^TB}$ and ideal $\ci_{w^{-1}}$, which they denote $I_w$.}

The following is a well-known characterization of matrix Schubert varieties and determinantal ideals.

\begin{prop} Let $w\in S_n$ and let $\ci_w$ denote its Schubert determinantal ideal. 
Then
\begin{enumerate}
\item $\cx_w^{\op}$ is the reduced subscheme $\Spec\ \!\C[\mathbf{z}]/\ci_w$; in particular, $\ci_w$ is prime, and
\item $\cx_w^{\op} \cap G = \pi^{-1}(X^{\op}_w).$ 
\end{enumerate}
\end{prop}
Using~\eqref{eq:oppositetoregularSchubert} and statement (1) above, we get that 
\begin{eqnarray}\label{eqn.Schubertscheme}
\cx_{w_0w}=\Spec\ \!\C[\mathbf{z}]/  w_0\cdot \ci_w
\end{eqnarray}
where $w_0\cdot \ci_w$ is the image of the Schubert determinantal ideal corresponding to $w$ under the action of $w_0$ on $\C[\mathbf{z}]$ defined by $w_0\cdot z_{ij} = z_{w_0(i)j}$.

Inclusions of matrix Schubert varieties define a partial order $\leq_{\Br}$ on the set of permutations known as \emph{Bruhat order}: $v\leq_\Br w$ if and only if $\cx_v \subseteq \cx_w$.  

Two families of Schubert ideals will play a large role in this paper. Define
\begin{equation}\label{eq.SchubertIdeals}
    \begin{aligned}
        \cj_i^0 &:= \left< z_{n1}, z_{n2}, \ldots, z_{ni} \right>\\
        \ck_i &:= \left<p_B \mid B\subseteq \{2, \ldots n\}, |B|=i \right>
    \end{aligned}
\end{equation}
for $i\in [n]$, where $p_B$ is the determinant of the submatrix $(z_{ k\ell})_{k\in B, 1\leq \ell\leq i}$. Observe that $\ck_n=\langle 0\rangle$. By convention, we assign $\cj_0^0 := \left< 0 \right>$ and $\ck_0=\C[\mathbf{z}]$.

For each $k,\ell\in [n], k\neq\ell$, let
\begin{itemize}
\item $u[k]$ be the shortest permutation $u$ in $S_n$ such that $u(k)=1$ and
\item $v[k]$ be the shortest permutation $v$ in $S_n$ such that $v(k)=n$.
\end{itemize}
For example, with $n=5$, $u[2] = [2 1 3 4 5]$ and $v[3] = [1 2 5 3 4]$. Note that $u[1] = v[n] = e$ is the identity.

The following lemmas are straightforward to verify using the definition of a matrix Schubert variety and the observation of Equation~\eqref{eqn.Schubertscheme}. 
\begin{lemma}\label{lem.Schubert} Let $i\in [n]$. The ideal $w_0\cdot \cj^0_{i-1}$ is the Schubert determinantal ideal for the permutation $u[i]$ and $w_0\cdot\ck_i$ is the Schubert determinantal ideal for the permutation $v[i]$. In particular,
\[
\cx_{u[i]}^{\op} = \Spec\ \!\C[\mathbf{z}]/(w_0\cdot\cj^0_{i-1})\qquad\mbox{and} \qquad\cx_{w_0v[i]} = \Spec\ \!\C[\mathbf{z}]/\ck_{i}
\]
for all $i=1,2,\dots, n$. Furthermore,  
when $i<j\leq n-1$, the ideal $w_0\cdot\cj_{i-1}^0 + w_0\cdot\ck_j$ is the Schubert determinantal ideal for the permutation $w[i,j]:=u[i]v[j]$, so that 
\[
\cx_{w_0w[i,j]} =  \Spec\ \!\C[\mathbf{z}]/(\cj_{i-1}^0 + \ck_j).
\]
\end{lemma}

Subvarieties of the flag variety equal to the intersection of a Schubert and opposite Schubert variety are called \emph{Richardson varieties}.  Let $u,v\in S_n$ and recall that the Richardson variety $X_u^{\op}\cap X_v$ is nonempty if and only if $u\leq_{\Br} v$ in Bruhat order.  In that case, we say that the closure $\overline{\pi^{-1}(X_u^{\op}\cap X_v)} \subseteq M_n(\C)$ is a \emph{matrix Richardson variety}. The next lemma shows that the sum of two of the  Schubert ideals considered above defines a matrix Richardson variety.

\begin{lemma} \label{lem.Schubert4} 
The ideal $w_0\cdot\cj_{i}^0 + \ck_{j}$ is prime for all $i, j \in [n]$ and $i=0$.  Furthermore, when $1\leq i \leq j\leq n$, the intersection 
\[
\cx_{u[i]}^\op \cap \cx_{w_0v[j]} = \Spec\ \!\C[\mathbf{z}]/ (w_0\cdot\cj_{i-1}^0 + \ck_{j})
\]
is a matrix Richardson variety. 
\end{lemma}

\begin{proof}  Note that $w_0\cdot \cj_0^0=\langle 0\rangle$ and otherwise $w_0\cdot \cj_i^0 = \left< z_{11}, z_{12}, \ldots, z_{1i} \right>$ and $\ck_j \subseteq \C[\mathbf{z}']$ where $\mathbf{z'} = \{z_{k\ell} \mid 2\leq k \leq n, 1\leq \ell \leq n\}$. Since $w_0\cdot \cj_i^0$ and $\ck_j$ are prime, and their generators have no common variables, the sum is prime. 

It is straightforward to check (using, for example, the tableau criterion from~\cite[Thm.~2.6.3]{Bjorner-Brenti}) that $i\leq j$ implies $u[i]\leq_{\Br} w_0v[j]$.  This implies $X_{u[i]}^{\op} \cap X_{w_0v[j]}$ is nonempty. By Lemma~\ref{lem.Schubert} the sum $w_0\cdot\cj_{i-1}^0 + \ck_{j}$ defines the affine subscheme $\cx_{u[i]}^{\op} \cap \cx_{w_0v[j]}$ in $\C[\mathbf{z}]$. Since $w_0\cdot\cj_{i-1}^0 + \ck_{j}$ is prime, this intersection is reduced and $$\overline{\pi^{-1}(X_{u[i]}^{\op} \cap X_{w_0v[j]})} = \cx_{u[i]}^{\op} \cap \cx_{w_0v[j]},$$ as desired.
\end{proof}

\subsection{Matrix Hessenberg schemes} \label{sec.Hess.schemes}
We now define our main geometric objects of interest, matrix Hessenberg schemes.   Throughout this section, $R$ denotes denoted a polynomial ring with coefficients in $\C$.  We take $R=\C$ and $R=\C[t]$ in future sections.  

\begin{definition} A function $h:[n]\to [n]$ is a \emph{Hessenberg function} if $h(i)\leq h(i+1)$ for all $i<n$ and $h(i)\geq i$ for all $i$.
\end{definition}

We denote a Hessenberg function by listing its values, e.g.~$h=(h(1), \ldots, h(n))$.  We say that $h$ is \emph{indecomposable} whenever $h(i)>i$ for all $i<n$. Otherwise, we say that $h$ is \emph{decomposable}. We adopt the convention that $h(0)=0$. 

Let $e_i$ denote the $n\times 1$ column vector with unique nonzero entry equal to $1$ in the $i$th row, and define 
\begin{eqnarray}\label{eqn.col}
v_j := \sum_{i=1}^n z_{ij}e_i 
\end{eqnarray}
for $j=1,\dots, n$.
Throughout this section, let $\mx\in M_n(R)$ denote a $n\times n$ matrix with entries in the polynomial ring $R$. 
For each $i\in [n]$ such that $h(i)<n$, define $\ci_{\mx,h, i}\subseteq R[\mathbf{z}]$ to be the ideal in $R[\mathbf{z}]$ generated by the condition that 
\begin{eqnarray}\label{eqn.rank.cond}
\mathrm{rk} \begin{bmatrix} | & | &  & | & | & | &  & | \\  \mx v_1 & \mx v_2 & \cdots & \mx v_i & v_1 & v_{2} & \cdots & v_{h(i)}\\ | & | &  & | & | & | &  & |
\end{bmatrix} \leq h(i).
\end{eqnarray}
That is, $\ci_{\mx,h, i}$ is the ideal generated by all $(h(i)+1)\times (h(i)+1)$-minors of the $n\times (h(i)+i)$-matrix above.  Set $\ci_{\mx,h, i}$ equal to the zero ideal whenever $h(i)=n$. 
We define
\begin{eqnarray}\label{eqn.hess-ideal}
\ci_{\mx, h} : = \sum_{i\in [n]} \ci_{\mx, h,  i}\subseteq R[\mathbf{z}].
\end{eqnarray}
Observe that $\ci_{\mx, h}$ is invariant under action of $B$ on $R[\mathbf{z}]$ induced by right multiplication on~$M_n(\C)=\Spec\, \C[\mathbf{z}]$.

Recall that $G=GL_n(\C)$ is a principal open subset of the affine scheme $M_n(\C)$.  We identify $G$ as an affine subscheme of $\Spec\ \!\C[\mathbf{z}, y]$ with coordinate ring $\C[\mathbf{z}, d^{-1}]:= \C[\mathbf{z}, y]/ \left< yd-1\right>$ obtained via localization, where $d\in \C[\mathbf{z}]$ denotes the determinant function. Let  
\begin{eqnarray}\label{eqn.GLn.coord}
\iota: R[\mathbf{z}]\hookrightarrow R\otimes_\C \C[\mathbf{z}, d^{-1}] =: R[\mathbf{z}, d^{-1}]
\end{eqnarray}
denote the canonical injection. Given an ideal $\ci\subseteq R[\mathbf{z}]$, 
we denote the saturation of $\ci$ with respect to $d$ by $\tilde{\ci}$, i.e.  
\begin{equation}\label{eq.definitionItilde}
    \tilde{\ci}:= \iota^{-1}(\iota(\ci)R[\mathbf{z}, d^{-1}]),
\end{equation} where $\iota(\ci)R[\mathbf{z}, d^{-1}]$ is  the ideal generated by the image of $\ci$ under $\iota$. It is immediate that $\ci\subseteq \tilde{\ci}$. We are now equipped to define matrix Hessenberg schemes; note that $R=\C$ in the definition below.  

\begin{definition} Given a matrix $\mx\in \mathfrak{gl}_n(\C)$ and Hessenberg function $h:[n]\to [n]$, the associated \emph{matrix Hessenberg scheme} is the affine scheme $$\fX_{\mx, h}:= \Spec\ \!\C[\mathbf{z}]/\tilde{\ci}_{\mx, h}.$$
\end{definition}

Note that, if $h$ is a fixed Hessenberg function, the matrix Hessenberg schemes associated to conjugate matrices are isomorphic~\cite[Rem.~3.5]{GP-semisimple}.  

\begin{rem}\label{rem:sameideal}
It is immediate from the definition that $\fX_{\mx, h}= \fX_{\mx', h}$  if the two ideals $\ci_{\mx, h}$ and $\ci_{\mx', h} $ are equal. Thus if $\mx$ and $\mx'$ result in identical rank conditions \eqref{eqn.rank.cond}, then the corresponding matrix Hessenberg schemes are equal. 
\end{rem}

In the literature, Hessenberg varieties are typically defined as subvarieties of the flag variety $G/B$ using the determinantal equations as defined for each $i\in [n]$ in~\eqref{eqn.rank.cond} above. There is some ambiguity regarding whether a Hessenberg variety is the subscheme of $G/B$ defined by these equations or the reduced variety supporting this subscheme (see~\cite[\S4]{Insko-Tymoczko-Woo} for more discussion on this topic).  We adopt the former convention in this paper (which is also the convention of~\cite{Insko-Tymoczko-Woo}).

Hessenberg schemes satisfy the following properties; see~\cite[\S2]{GP-semisimple} for a detailed discussion. 

\begin{rem}\label{rem.quotient} Given $\mx\in \mathfrak{gl}_n(\C)$ and  $h: [n] \to [n]$ a Hessenberg function, let  $Y_{\mx, h}$ denote corresponding the Hessenberg variety in the flag variety $G/B$.  Then
\begin{enumerate}
\item $\pi^{-1}(Y_{\mx, h}) = G\cap \Spec\ \!\C[\mathbf{z}] / \ci_{\mx, h}$, and
\vspace*{.08in}
\item $\fX_{\mx, h} = \overline{\pi^{-1}(Y_{\mx, h})}$.  
\end{enumerate}
\end{rem}

We say $i\in[n]$ is a \emph{corner of the Hessenberg function $h$} whenever $h(i)>h(i-1)$. Denote the set of all corners by $\cc(h)$. Our convention that $h(0)=0$ implies that $1$ is always a corner of $h$. For each $i\in [n]$ we set
\[
i^*:= \max\{ k\in [n] \mid h(k)=h(i) \}.
\]
Notice that if $i$ is a corner of $h$, then $i^*+1 \in \cc(h)$. Furthermore, there is no corner $k$ in $\cc(h)$ such that $i<k<i^*+1$. The following statement is proved in~\cite[Lemma 3.6]{GP-semisimple}; note that version stated in~\cite{GP-semisimple} assumes $R=\C$, but the proof generalizes immediately to the polynomial setting.

\begin{lemma}\label{lem.corners} Let $R$ be a polynomial ring with coefficients in $\C$. For all Hessenberg functions $h$ and matrices $\mx \in M_n(R)$,
\[
\ci_{\mx, h} = \sum_{i\in \cc(h)} \ci_{\mx, h, i ^*}.
\]
\end{lemma}

\begin{example}\label{ex.2444.2}We use Lemma~\ref{lem.corners} to illustrate that a nilpotent matrix Hessenberg scheme $\fX_{\nilp, h}$ need not  be equidimensional. Let $h=(2,4,4,4)$, which is indecomposable. The corner set of $h$ is $\cc(h)=\{1,2\}$, and we have   $1^*=1$ and $2^*=4$. Recall that $\ci_{\mx, h, 4} = \left< 0 \right>$ by definition. Thus Lemma~\ref{lem.corners} implies $\ci_{\mx, h} = \ci_{\mx, h, 1}$. 

Let $\mx = \nilp:=E_{14}$ be the $4\times 4$ nilpotent matrix with unique nonzero entry equal to $1$ in position $(1,4)$. Note that $\nilp$ is a nilpotent element in the minimal sheet $\fg_{(2,1,1)}$. 
The rank conditions from \eqref{eqn.rank.cond} generating $\ci_{\nilp,h}$ are
\begin{eqnarray}
\rk\begin{bmatrix} 
z_{41} &  z_{11} &  z_{12}\\ 
0	&  z_{21} &  z_{22}\\
0	&  z_{31} &  z_{32}\\
0	&  z_{41} &  z_{42}\\
 \end{bmatrix}\leq 2,
\end{eqnarray}
which is equivalent to insisting that all $3\times 3$ minors vanish. Thus the ideal $\ci_{\nilp, h}$ has generators $z_{41}(z_{21} z_{32}- z_{22} z_{31})$, $z_{41}(z_{21} z_{42}- z_{22} z_{41})$ and $z_{41}(z_{31} z_{42}- z_{32} z_{41})$.

We use Macaulay2 to check that $\ci_{\nilp, h}$ is radical, and has associated primes
\begin{align*}
\ck_2 &=\langle z_{32}z_{41}-z_{31}z_{42},z_{22}z_{41}-z_{21}z_{42},z_{22}z_{31}-z_{21}z_{32} \rangle\quad \mbox{and} \\
\cj_1^0 &=\langle z_{41} \rangle.
\end{align*}
We can identify each prime ideal as the determinantal ideal defining a matrix Schubert variety using Lemma~\ref{lem.Schubert}. The two relevant permutations in this case are $v[2] = [1423]$ and $u[2]$ = [2134].
Observe that the vanishing set of both prime ideals intersects $GL_4(\C)$, so $\tilde\ci_{\nilp, h} = \ci_{\nilp, h}$. In Section~\ref{se:AssociatedPrimesMinSheet}, we show that $\tilde\ci_{\mx, h} = \ci_{\mx, h}$ for any $h$ when $\mx$ is in the minimal sheet. For $h=(2,4,4,4)$, we conclude that $\fX_{\nilp, h} $ is the scheme theoretic union of matrix Schubert varieties:
$$
\fX_{\nilp, h} = \cx_{w_0[1423]}\cup \cx_{w_0[2134]} = \cx_{[4132]}\cup \cx_{[3421]}.
$$
Since $ \dim_\C \cx_{[4132]}=14$ and $\dim_\C \cx_{[3421]}=15$, we have shown that the reduced (but not irreducible) matrix Hessenberg scheme $\fX_{\nilp, h}$ is \emph{not equidimensional}.

The interested reader can compare this with~\cite[Ex.~3.9]{GP-semisimple} which considers the matrix Hessenberg scheme for the semisimple element $\semi = (1,0,0,0)\in \fg_{(2,1,1)}$ and where it is shown that for $h=(2,4,4,4)$ we have
\[
\fX_{\semi, h} = \cx_{[4132]}\cup \cx_{[2134]}^{\op}.
\]
In particular, we note that $\dim \cx_{[4132]} = 14$ and $\dim \cx_{[2134]}^{\op} = 15$ so $\fX_{\nilp,h}$ and $\fX_{\semi,h}$ have the same number of components of a given dimension.  This phenomenon is explained by the existence of a flat degeneration from each component of $\fX_{\semi, h}$ to a corresponding component of $\fX_{\nilp, h}$; see Theorem~\ref{thm.components.indecomp} below. 
\end{example}


\section{Sheets and Families of matrix Hessenberg schemes} \label{sec.sheets}

In this section, we consider the stratification of the Lie algebra $\mathfrak{gl}_n(\C)$ by sheets.  We make Question~\ref{quest1} more precise by defining, for each matrix $\mx$, a curve in the sheet containing $\mx$ and family of matrix Hessenberg schemes with general fiber $\fX_{\mx, h}$ and special fiber equal to a nilpotent matrix Hessenberg scheme. Our main result, stated precisely in Theorem~\ref{thm.flat} below, is that this family is flat whenever $\mx$ is an element of the minimal sheet.

\subsection{Sheets}
Given an arbitrary complex reductive group $G$, we consider the adjoint action of $G$ on its Lie algebra $\fg$.  The \emph{sheets of $\fg$} are irreducible components of the locally closed $G$-stable subvariety $\fg^{(d)} := \{ x\in \fg \mid \dim Gx = d\}$. That is, each sheet in $\fg$ is a maximal irreducible subset of $\fg$ consisting of $G$ orbits (with respect to the adjoint action) of a fixed dimension.  This decomposition was first considered by Peterson for $GL_n(\C)$~\cite{Peterson-thesis} and  Borho~\cite{Borho1981} classified the orbits appearing in each sheet of the corresponding Lie algebra. The study of sheets arises naturally in the study of orbits and representation theory of algebraic groups~\cite{Borho-Kraft1979, Katsylo1982, Panyushev2001}. Although sheets are defined for arbitrary reductive groups, we focus here on $G=GL_n(\C)$ and leave the study of Question~\ref{quest1} and matrix Hessenberg schemes in other types for future work.

Kraft proved that the sheets in $\mathfrak{gl}_n(\C)$ are in bijection with partitions of $n$~\cite[\S2]{Kraft}. Recall that a \emph{partition} $\lambda$ of $n$ is a list $\lambda=(\lambda_1,  \lambda_2, \ldots , \lambda_\ell)$ of positive integers such that $\lambda_1\geq \lambda_2 \geq \cdots \geq \lambda_\ell$ and  $n=\lambda_1+\cdots+\lambda_\ell$. We associate a partition to $\mx\in \mathfrak{gl}_n(\C)$ as follows.
Recall that a generalized Jordan block of the matrix $\mathsf{x}\in \mathfrak{gl}_n(\C)$ consists of all Jordan blocks of $\mathsf{x}$ with the same eigenvalue $c\in \C$, and each generalized Jordan block determines a partition $\mu$
recording the sizes of these Jordan blocks in decreasing order. Let $\mu_j$ denote the $j$th entry of $\mu$.
The algebraic multiplicity of the eigenvalue $c$ is given by $|\mu|=\sum_j \mu_j$.

Let $\{c_1, c_2,\dots, c_k\}$  be the distinct eigenvalues of $\mx$ and denote the partition corresponding to $c_i$ by $\mu^i$. 
Define $\lambda_{\mx,j}:= \sum_{i=1}^k \mu^{i}_j$. Observe that $\lambda_{\mx,j}$ does not depend on the order of the eigenvalues chosen, and $\lambda_\mx:= (\lambda_{\mx,1}, \lambda_{\mx, 2}, \ldots)$ is a partition of $n$. For each partition $\lambda$ the sheet
$\fg_\lambda$ is the union of orbits $\co_\mx$ such that $\lambda_\mx=\lambda$.

\begin{rem}
If $\mx$ is nilpotent, all its eigenvalues are $0$, resulting in only one generalized Jordan block with partition $\mu$. In this case $\lambda_\mx = \mu$.  Thus each sheet $\fg_\lambda$ contains a unique nilpotent conjugacy class, namely the conjugacy class associated to the partition $\mu$.
\end{rem}

\begin{example} \label{matrix.ex} Consider the matrix
\begin{eqnarray*}
\mx =\begin{bmatrix} 4 & 0 &0 & 0 & 0 & 0 \\ 0 & 4 & 0 & 0 & 0 & 0\\ 0 & 0 & 5 & 0 & 0 & 0\\ 0 & 0 & 0 & 5 & 1 & 0\\ 0 & 0 & 0 & 0 & 5 & 1\\ 0 & 0 & 0 & 0 & 0 & 5
\end{bmatrix} 
\end{eqnarray*} 
with eigenvalues $c_1=4$ and $c_2=5$.
Then $\mu^1=(1,1)$ and $\mu^2=(3,1)$, resulting in $\lambda_\mx = (1+3, 1+1) = (4,2)$ so $\mx\in \fg_{(4,2)}$. 
\end{example}

\begin{definition}
The \emph{regular sheet} $\fg_{reg}:= \fg_{(n)}$ consists of all  orbits of dimension $n^2 - n$.
The \emph{minimal sheet} $\fg_{(2,1^{n-2})}$ consists of orbits of minimal nonzero dimension. In particular, $\fg_{(2,1^{n-2})}$ contains the unique nilpotent orbit of minimal nonzero dimension in $\mathfrak{gl}_n(\C)$, called the \emph{minimal nilpotent orbit}. 
\end{definition}

\begin{rem}\label{rem.minimal} The minimal sheet $\fg_{(2,1^{n-2})}$ contains only two types of conjugacy classes.
If $\lambda_\mx=(2,1^{n-2})$, then $\mx$ has at most two eigenvalues.  If $\mx$ has exactly two eigenvalues $c_1$ and $c_2$ then  $\mu^{1} = (1)$ and $\mu^{2}=(1,1,\ldots, 1)$. Thus $\mx$ is conjugate to $\diag(c_1, c_2, \ldots, c_2)$.  If $\mx$ has a single eigenvalue, then $\mu^{1} =(2,1,\ldots, 1)$, and $\mx$ is conjugate to $cI_n +E_{n-1,n}$.
\end{rem}

\begin{lemma}\label{lem:equality} Up to isomorphism there are only two matrix Hessenberg schemes over the minimal sheet, namely $\fX_{\semi, h}$ and $\fX_{\nilp ,h}$ where $\semi=\diag(1,0,\dots,0)$ and $\nilp := E_{1,n}$. 
\end{lemma}
\begin{proof} By Remark~\ref{rem.minimal} there are only two types of matrices in $\fg_{(2,1^{n-2})}$, up to a choice of eigenvalues.  If $\mx = \diag(c_1, c_2,\dots,c_2)$ for some $c_1,c_2\in \C$ with $c_1\neq c_2$ then the rank inequalities of \eqref{eqn.rank.cond} are satisfied for $\mx$ if and only if they are satisfied for $\mx' = \diag(c_1-c_2,0, \dots, 0)$ which are themselves satisfied if and only if they are satisfied for $\semi = \diag(1, 0,\dots, 0)$. It follows from Remark~\ref{rem:sameideal} that, for any Hessenberg function $h$, 
$\ci_{\mx,h}=\ci_{\semi,h}$. If $\mx = cI_n + E_{n-1,n}$ for some $c\in \C$, then the rank inequalities of \eqref{eqn.rank.cond} are satisfied for $\mx$ if and only if they are satisfied for $\mx'=E_{n-1,n}$ and thus $\ci_{\mx, h} = \ci_{\mx',h}$ by Remark~\ref{rem:sameideal}. The desired statement now follows from the fact that $\mx'$ and $\nilp$ are conjugate.
\end{proof}

For any $\mx\in \fg_\lambda$, we  identify a specific conjugate in Jordan canonical form. 
Fix an ordering of the generalized Jordan blocks as follows. Let $\{c_1, c_2,\dots, c_k\}$  denote the distinct eigenvalues of $\mx$. By relabeling, we may assume $i<j$ whenever 
\begin{itemize}
    \item $|\mu^{i}|< |\mu^{j}|$, or 
    \item $|\mu^{i}|= |\mu^{j}|$, and $\mu^{i}$ is less than $\mu^{j}$ in lexicographic ordering.
\end{itemize}

For each eigenvalue $c_i$, consider the unique tableau $\ct_{i}$ obtained by labeling the Young diagram of shape $\mu^{i}$ with entries
\begin{eqnarray*}
|\mu^{1}|+ \cdots + |\mu^{{i-1}}|+1,\,|\mu^{1}|+ \cdots + |\mu^{{i-1}}|+2,\,\ldots,\, |\mu^{1}|+ \cdots + |\mu^{{i-1}}| +|\mu^{i}| 
\end{eqnarray*}
so that all values \emph{increase} from left to right in each row, and so that all values in each row are \emph{greater than} all values in the rows below it (that is, the tableau $\ct_i$ has decreasing columns). Note that this  tableau is obtained by labeling the rows sequentially from left to right, starting with the bottom row and moving upward; see Example~\ref{ex.fill} below. Let
\[
\nilp_i = \sum E_{\ell m}
\] 
where the sum is taken over all pairs $(\ell,m)$ such that $m$ appears directly to the right of the box containing $\ell$ in $\ct_{i}$. 

Then
$\mx$ is conjugate to the matrix 
$\mx_{\mathrm{ss}}+\mx_{\mathrm{nilp}}$ with semisimple and nilpotent parts
\begin{equation}
\begin{aligned}\label{eqn.semi.part}
\mx_{\mathrm{ss}}&:= \diag(c_1, \ldots, c_1, c_2, \ldots, c_2, \ldots, c_k, \ldots, c_k)\\
\mx_{\mathrm{nilp}} &:=\nilp_1+\nilp_2+\cdots+\nilp_k,
\end{aligned}
\end{equation}
where $c_i$ occurs $|\mu^{i}|$ times as a diagonal entry of $\mx_{\mathrm{ss}}$. 

\begin{rem}\label{rem.decomp.JCF} If $\mx$ is in Jordan canonical form with blocks arranged according to $(c_1, \dots, c_k)$, and such that the nilpotent part within each generalized Jordan block has {\em increasing} block sizes, then $\mx=\mx_{\mathrm{ss}}+\mx_{\mathrm{nilp}}.$ The matrix in Example~\ref{matrix.ex} is of this form. 
\end{rem}

Finally, we identify a particular nilpotent element $\nilp_\mx$ that lies in the same sheet as $\mx.$ 
Given two tableaux $\ct$ and $\ct'$, let $\ct\circ \ct'$ denote their \emph{concatenation},  the tableau obtained by concatenating the entries in each row of $\ct$ and $\ct'$.  Let $\ct_\mx$ be tableau of shape $\lambda_\mx$ obtained by the concatenation:
$$
\ct_\mx = \ct_{1}\circ \ct_{2} \circ \cdots \circ \ct_{k}.
$$

\begin{example}\label{ex.fill} Consider the matrix $\mx\in \fg_{(4,2)}$ as in Example~\ref{matrix.ex}. We have
\[
\ct_1 =   \begin{ytableau}
2\\
1 
\end{ytableau} \;\; \text{ and } \;\; 
\ct_2 =   \begin{ytableau}
4 & 5 & 6\\
3 
\end{ytableau}
\]
and the tableau $\ct_\mx$ is given by
\[
\begin{ytableau}
2\\
1 
\end{ytableau} 
\circ
\begin{ytableau}
4 & 5 & 6\\
3 
\end{ytableau}
=
 \begin{ytableau}
 2 & 4 & 5 & 6\\
 1 & 3
\end{ytableau} \ .
\]   
\end{example}

We use the tableau $\ct_\mx$ to define a nilpotent matrix:
\begin{eqnarray}\label{eqn.nilp.def}
\nilp_\mx:= \sum E_{\ell m}
\end{eqnarray}
where the sum is taken over pairs $(\ell, m)$ such that $m$ appears in a box directly to the right of the box containing $\ell$ in $\ct_{\mx}$. We call $\nilp_\mx$ the \emph{associated nilpotent element of $\mx$.} Note that $\nilp_\mx\neq \mx_{\mathrm{nilp}}$ in general.
Since $\ct_\mx$ has shape $\lambda_\mx$,  the nilpotent matrix $\nilp_\mx$ is in the same sheet as $\mx$. 
 When $\mx$ is nilpotent, the nilpotent matrix $\nilp_\mx$ is conjugate to $\mx$. Example~\ref{matrix.ex.2} below computes $\nilp_\mx$ for the matrix $\mx$ of Example~\ref{matrix.ex}.


\subsection{Families of matrix Hessenberg schemes over arbitrary sheets} 
For $\mx\in \fg_{\lambda}$, we construct a curve $\mx_t$ that lies in $\fg_\lambda$ and contains a conjugate of $\mx$ and the associated nilpotent element $\nilp_\mx$. We then define a
family of matrix Hessenberg schemes lying over $\mx_t$ through which we will degenerate Hessenberg schemes.

\begin{definition}\label{def.sheetline}
    Suppose $\mx\in \fg_{\lambda}$ is conjugate to $\mx_{\mathrm{ss}}+\mx_{\mathrm{nilp}}$, where the latter is in Jordan canonical form with Jordan blocks arranged as in Remark~\ref{rem.decomp.JCF}. 
    Consider the path in $M_n(\C)$ given by
\begin{eqnarray} \label{eqn.xt.def}
\mx_t := t\,\mx_{\mathrm{ss}} + \nilp_\mx, \quad t\in \C
\end{eqnarray}
where $\nilp_\mx \in \fg_\lambda$ is the associated nilpotent element of $\mx$ from \eqref{eqn.nilp.def}.
We call $\mx_t$ the \emph{sheet line associated with $\mx$.} If $\mx$ is nilpotent, $\mx_t= \nilp_\mx$ and the sheet line is a single point.
\end{definition}

The astute reader will notice that $\mx$ is not necessarily on the sheet line $\mx_t$, even when $\mx= \mx_{\mathrm{ss}}+\mx_{\mathrm{nilp}}$. However, $\mx$ is conjugate a point on the line, and furthermore the entire sheet line lives in the sheet $\fg_\lambda$. We leave the proof of the following lemma to the reader. 
\begin{lemma}\label{lemma.conjugate} For all $a\in \C^*$, the matrix $\mx_a$ is conjugate to $a\mx$.   In particular, if $\mx\in \fg_{\lambda}$ then $\mx_a \in \fg_{\lambda}$ for all $a\in \C$.
\end{lemma}

\begin{example}\label{matrix.ex.2} Using $\mx$ from Example ~\ref{matrix.ex} and the tableau from Example~\ref{ex.fill} we have, 
\begin{eqnarray*}
\nilp_\mx = \begin{bmatrix} 0 & 0 & 1 & 0 & 0 & 0\\ 0 & 0 & 0 & 1 & 0 & 0\\ 0 & 0 & 0 & 0 & 0 & 0 \\ 0 & 0 & 0 & 0 & 1 & 0\\ 0 & 0 & 0 & 0 & 0 & 1\\0 & 0 & 0 & 0 & 0 & 0 \end{bmatrix} 
\quad \mbox{ and } \quad \mx_t = \begin{bmatrix} 4t & 0 & 1 & 0 & 0 & 0\\ 0 & 4t & 0 & 1 & 0 & 0\\ 0 & 0 & 5t & 0 & 0 & 0  \\ 0 & 0 & 0 & 5t & 1 & 0\\ 0 & 0 & 0 & 0 & 5t & 1\\ 0 & 0 & 0 & 0 & 0 & 5t \end{bmatrix} . 
\end{eqnarray*}
\end{example}

\begin{example} We illustrate the calculation of the sheet line when $\mx\in \fg_{(n)}$ is a regular semisimple matrix. 
We may assume $\mx = \diag(c_1, c_2, \dots, c_n)$ is a diagonal matrix, with $c_i\neq c_j$ for $i\neq j$. The Jordan blocks of $\mx$ are all $1\times 1$ matrices and $\mu^{i} = (1)$ for all $i$. It follows that the associated partition to $\mx$ is $\lambda_\mx = (n)$, while the tableau $\ct_\mx$ is the following. 
\ytableausetup
 {mathmode, boxframe=normal, boxsize=1.6em}
\[
 \begin{ytableau}
 1 & 2 & 3 & \none[\cdots] & n
\end{ytableau}
\]
Thus $\nilp_\mx = E_{12}+E_{23}+\cdots E_{n-1, n}$, and
$$
\mx_t =
\begin{bmatrix}
tc_1 & 1 & 0 &\dots &0\\
0& tc_2 & 1& \dots &0\\
\vdots & \vdots & \vdots &\ddots &0\\
0 & 0 & 0 &\dots & 1 \\
0 & 0 & 0&  \dots & tc_n 
\end{bmatrix}
$$
is the associated sheet line.
\end{example}

\begin{example}\label{ex.min.sheet.x}  We illustrate the calculation of the sheet line when $\semi= \diag(1,0,\ldots, 0)$ is a semisimple element of the minimal sheet $\fg_{(2,1^{n-2})}$. 
 In this case,
$\mu^{1}=(1)$ and $\mu^{2}= (1^{n-1})$.
The corresponding partition is $\lambda_\mx = (2, 1, 1,\dots, 1)$, and the tableau $\mathcal{T}_\mx$ is
$$
\ytableausetup{mathmode, boxframe=normal, boxsize=1.6em, centertableaux}
\begin{ytableau}
1 \\
\end{ytableau}
\;
\circ
\;\begin{ytableau}
n \\
\scriptstyle{n-1}  \\
\none[\vdots] \\
3\\
2
\end{ytableau} \; 
=
 \;\begin{ytableau}
 1 & n \\
\scriptstyle{n-1}  \\
\none[\vdots] \\
3\\
2
\end{ytableau}
$$
so the associated nilpotent element of $\semi$ is $E_{1n}$. 
The sheet line~\eqref{eqn.xt.def} is given by 
\begin{eqnarray}\label{eqn.xt.min}
\semi_t:=  t\semi + E_{1n} = \begin{bmatrix} t  & 0 & \cdots &   1 \\ 0  & 0& \cdots &  0  \\ \vdots & \vdots & \ddots  & \vdots\\  0  & 0 & \cdots  & 0
\end{bmatrix}.
\end{eqnarray}
\end{example}

We now define a family of matrix Hessenberg schemes defined over the sheet line $\mx_t$.
\begin{definition}\label{def.FamilyOverSheetLine}
Let 
\[
\fX_{\mx_t, h} := \Spec\ \!\C[t, \mathbf{z}]/ \tilde{\ci}_{\mx_t, h},
\]
where $\ci_{\mx_t, h} \subseteq \C[t, \mathbf{z}]$ is the ideal defined in Equation~\eqref{eqn.hess-ideal} (with $R=\C[t]$) and $\tilde{\ci}_{\mx_t, h}$ is the saturation of $\ci_{\mx_t, h}$ with respect to the determinant $d$ of $Z=(z_{ij})$. The \emph{family of matrix Hessenberg schemes over the sheet line $\mx_t$} is the morphism
\begin{equation}\label{eqn.morphism}
\varphi_{\mx, h}: \fX_{\mx_t, h} \longrightarrow \Spec\ \!\C[t]
\end{equation}
induced from the ring homomorphism $\C[t]\rightarrow \C[t,{\bf z}]/\tilde{\ci}_{\mx_t, h}$. By a slight abuse of notation, we also refer to the scheme $\fX_{\mx_t, h}$ as the \emph{family of matrix Hessenberg schemes} over the sheet line $\mx_t$.
\end{definition}

\begin{lemma}\label{lemma.fibers} For each $\mx\in \fg$ and Hessenberg function $h: [n]\to [n]$, the morphism $\varphi_{\mx, h}$ has scheme-theoretic fiber over $a\neq 0$ isomorphic to the matrix Hessenberg variety $\fX_{\mx, h}$ and scheme-theoretic fiber over $0$ equal to the nilpotent matrix Hessenberg scheme~$\fX_{\nilp_\mx, h}$,  where $\nilp_\mx$ is the associated nilpotent element to $\mx$. 
\end{lemma}
\begin{proof} For each $a\in \C$ the fiber of $\varphi_{\mx, h}$ over the prime ideal $\langle t-a\rangle\in \Spec\ \!\C[t]$ consists of prime ideals $\cp\subseteq \C[t,\mathbf{z}]$ such that $\tilde{\ci}_{\mx_t, h}+\langle t-a \rangle \subseteq \cp$.  Thus
\[
\varphi_{\mx, h}^{-1}(\langle t-a\rangle) = \Spec\ \!\C[t,\mathbf{z}]/ (\tilde{\ci}_{\mx_t, h}+\langle t-a\rangle).
\]
Let $\mathrm{ev}_a: \C[t,\mathbf{z}] \to \mathbf{\C[z]}$ denote the evaluation homomorphism mapping $t$ to $a$.  Note that $\mathrm{ev}_a$ satisfies $\ker(\mathrm{ev}_a) = \left< t-a \right>$ and $\mathrm{ev}_a(\ci_{\mx_t, h}) = \ci_{\mx_a, h}$.  Thus $\mathrm{ev}_a$ induces a bijection between between prime ideals in $\C[t,\mathbf{z}]$ containing ${\ci}_{\mx_t, h}+\langle t-a\rangle$ and prime ideals in $\C[\mathbf{z}]$ containing $\ci_{\mx_a, h}$.  The fact that $\mathrm{ev}_a(d)=d$ yields an isomorphism of schemes, 
\[
 \Spec\ \!\C[t,\mathbf{z}]/ (\tilde{\ci}_{\mx_t, h}+\langle t-a\rangle) \simeq \Spec\ \!\C[\mathbf{z}]/ \tilde{\ci}_{\mx_a, h}
\]
so each fiber of $\varphi_{\mx, h}$ is a matrix Hessenberg scheme.
If $a\neq 0$, the matrix $\mx_a$ is conjugate to $a\mx$ by Lemma~\ref{lemma.conjugate} and thus we get an isomorphism between the fiber $\fX_{\mx_a ,h}$ over $a$ and $\fX_{\mx,h}$ in that case. Finally, by construction $\mx_0 = \nilp_\mx$ so the fiber over $0$ is  $\fX_{\nilp_\mx, h}$. 
\end{proof}

Using the specific family from~\eqref{eqn.morphism}, we conjecture that the answer to Question~\ref{quest1} is yes.

\begin{conj}\label{conj}  The one-parameter family $\varphi_{\mx,h}$ is a flat morphism for all $\mx\in \mathfrak{gl}_n(\C)$. 
\end{conj}
   
The main result of this paper proves this conjecture for $\mx$ in the minimal sheet.

\begin{thm}\label{thm.flat} For all Hessenberg functions $h: [n]\to [n]$ and all $\mx\in \fg_{(2,1^{n-2})}$, the morphism $\varphi_{\mx, h}: \fX_{\mx_t,h} \to \Spec\ \!\C[t]$ is flat.  In particular, every matrix Hessenberg scheme over the minimal sheet admits a flat degeneration to the minimal nilpotent matrix Hessenberg scheme.
\end{thm}

From the theorem, it follows that all matrix Hessenberg schemes over the minimal sheet have the same dimension and cohomology classes (see Lemma~\ref{lem.flat} and Remark~\ref{rem.equalcohomologyclasses} below). We explore these consequences after the proof of Theorem~\ref{thm.flat} in Section~\ref{sec.FlatMinSheet}.

The following is an example of the family $\varphi_{\mx, t}$ over the minimal sheet, which is flat by Theorem~\ref{thm.flat}.

\begin{example}\label{ex.flat.h2444} Consider the scheme $\fX_{\mx_t,h}$ associated with $\semi = \mathrm{diag}(1, 0,0,0)\in \fg_{(2,1,1)}$ for $h=(2,4,4,4)$. We describe the family explicitly by identifying the associated ideals of $\tilde{\ci}_{\mx_t,h}$, and illustrate that $\fX_{\mx_t,h}$ is not equidimensional.

We noted in Example~\ref{ex.min.sheet.x} that 
\[
\mx_t=\semi_t = \mathrm{diag}(t,0,0, 0) + E_{14}.
\]
As in Example~\ref{ex.2444.2}, the ideal $\ci_{\mx_t, h}$ is a sum over two corner elements, one of which vanishes, so  $\ci_{\mx_t, h}= \ci_{\mx_t, h, 1}$.
The rank conditions defining $\ci_{\mx_t, h}$ are then 
\begin{eqnarray*}
\rk\begin{bmatrix} 
tz_{11}+z_{41} &  z_{11} &  z_{12}\\ 
0	&  z_{21} &  z_{22}\\
0	&  z_{31} &  z_{32}\\
0	&  z_{41} &  z_{42}\\
 \end{bmatrix}\leq 2.
\end{eqnarray*}
Thus,
\begin{eqnarray*}
\ci_{\mx_t, h} &=& \left< tz_{11}+z_{41} \right> \cdot \left< z_{21}z_{32}-z_{22}z_{31},z_{21}z_{42}-z_{22}z_{41}, z_{31}z_{42}-z_{32}z_{41}  \right>\\
&=& \left< tz_{11}+z_{41} \right> \cdot \ck_2.
\end{eqnarray*}
Using Macaulay2, we confirm that 
\[
\ci_{\mx_t, h} = \ck_2 \cap \left<  tz_{11}+z_{41} \right>
\]
and it is clear that each of the ideals $\ck_2$ and $\left< tz_{11}+z_{41} \right>$ is prime. Since the vanishing set of each of these ideal intersects $GL_4(\C)$, we have $\tilde{\ci}_{\mx_t, h} = \ci_{\mx_t, h}$.  We conclude that $\fX_{\mx_t, h}$ is the reduced scheme theoretic union of two irreducible components, one of dimension $\dim \C[t,\mathbf{z}]/\ck_2=15$ and the other of dimension $\dim \C[t,\mathbf{z}]/\left< tz_{11}+z_{41} \right>=16$. In particular, the scheme $\fX_{\mx_t, h}$ is \emph{not equidimensional}.

The fiber of the family $\varphi_{\mx, h}$ over $0$ is precisely the nilpotent matrix Hessenberg scheme of Example~\ref{ex.2444.2}. The fiber over $a\neq 0$ is isomorphic to the semisimple matrix Hessenberg scheme $\fX_{\semi, h}$. Thus we obtain a flat degeneration from $\fX_{\semi, h}$ to $\fX_{\nilp, h}$. 
\end{example}

Our proof of Theorem~\ref{thm.flat} leverages commutative algebra in the polynomial ring $\C[\mathbf{z}, t]$ to study the ideal $\tilde{\ci}_{\mx_t, h}$ defining the scheme $\fX_{\mx_t, h}$.
Theorem~\ref{thm.assprimes} below identifies the associated primes of the family over the minimal sheet for any Hessenberg function.


\section{Flat morphisms and miracle flatness}\label{sec.flat}

In this section, we review properties of flatness, setting the stage for the proof that $\varphi_{\mx, h}:\fX_{\mx_t, h}\rightarrow \Spec\ \!\C[t]$ is flat when $\mx$ is in the minimal sheet. We explain why standard techniques used in some other cases (miracle flatness) do not apply in the setting of the minimal sheet.

\subsection{Flat morphisms}
Suppose $R$ and $S$ are commutative rings with identity.
Recall that a morphism $\varphi: \Spec\ \! R \to \Spec\ \!S$ of affine schemes is \emph{flat} whenever $R$, regarded as a $S$-module via the map $\varphi^*: S\to R$, is a flat $S$-module.  
In that case, we say that the collection of fibers of $\varphi$ is a \emph{flat family}.  In this paper we assume $S=\C[t]$, in which case the flat morphism $\varphi$ is called a \emph{flat degeneration}. Let $a\neq 0$ be a complex number. We call any scheme isomorphic to the scheme-theoretic fiber over the closed point $\left< t-a \right>\in \Spec\ \!\C[t]$ the \emph{general fiber} of $\varphi$ and the scheme-theoretic fiber over $0$ the \emph{special} or \emph{degenerate fiber}.

A particularly nice criterion for flatness applies for families defined over a principal ideal domain, captured in the following lemma; see~\cite[Cor~6.3, pg.~164]{Eisenbud}. We use it in Section~\ref{sec.FlatMinSheet} to prove Theorem~\ref{thm.flatfamily}.

\begin{lemma}\label{lemma.torsion} Let $S$ be a commutative ring and $M$ an $S$-module. If $S$ is a principal ideal domain, then $M$ is flat as an $S$-module if and only if $M$ is torsion free. 
\end{lemma}

Flat families preserve several invariants of the corresponding fibers, resulting in relationships among the fibers. We collect them here for ease of reference.

\begin{lemma}\label{lem.flat} Suppose $\varphi: \fX_t \to \Spec\ \!\C[t]$ is a flat 
family with general fiber $\fX$ and special fiber $\fX_0$.  
Then 
\begin{enumerate}
\item $\dim \fX = \dim \fX_0$, and
\item if $\fX_t$ is a graded scheme, then the {multidegree} of $\fX$ and $\fX_0$ with respect to this grading are equal.
\end{enumerate}
\end{lemma}

These properties stem from the observation that graded flat families preserve {\em the graded Hilbert series} of the fibers~\cite[Exerercise 6.11, pg.~175]{Eisenbud}.  
In particular, when an ideal $\ci$ of $\C[\mathbf{z}]$ is homogeneous with respect to a $\mathbb Z^d$ grading of the polynomial ring, the quotient $\C[\mathbf{z}]/\ci$ is has a well defined graded Hilbert series, which depends on the grading. In this case, we associate the graded Hilbert series with the scheme $\fX= \Spec\ \!\C[\mathbf{z}]/\ci$.

We  refer the interested reader to~\cite[\S 8.1,8.2]{MS05} for a formal definition of the graded Hilbert series. Importantly, the graded Hilbert series determines a scheme's \textit{multidegree}, which itself determines the scheme's dimension and cohomology class.   The multidegree is a polynomial in $\Z[\mathbf{x}] = \Z[x_1, x_2, \ldots, x_d]$, of degree $\dim (\fX)$, defined using a particular shift and truncation of the graded Hilbert series.  A precise definition can be found in~\cite[\S 8.5]{MS05}.

Suppose $\fX$ is a closed subscheme of the affine scheme $M_n(\C)=\Spec\ \!\C[\mathbf{z}]$ of $n\times n$ matrices such that $\fX \cap G = \pi^{-1}(Y)$, where $Y\subseteq G/B$ is a closed subscheme. Knutson and Miller~\cite{Knutson-Miller2005} proved that the multidegree of the coordinate ring of $\fX$ (with respect to the grading $\deg(z_{ij})= e_j,$ the $j$th coordinate vector) is equal to the class of $Y$ in the cohomology ring of the flag variety $G/B$ (see also~\cite[Prop.~3]{Insko-Tymoczko-Woo}).  Here the cohomology ring $H^*(G/B)$ has been identified with a quotient of the polynomial ring $\Z[x_1,\dots, x_n]$ via the Borel presentation.  In summary, the cohomology classes of fibers in a graded flat family are equal. We summarize this conclusion in the remark below.

\begin{rem}\label{rem.equalcohomologyclasses}
Suppose that $\fX_t\subseteq \Spec\ \!\C[t,\mathbf{z}]$, and that
$\varphi: \fX_t\rightarrow \Spec\ \! \C[t]$ is a flat family with general fiber $\fX$ and special fiber $\fX_0$. Let $Y$ and $Y_0$ be closed subschemes of the flag variety $G/B$ such that $\pi^{-1}(Y)=\fX\cap G$ and $\pi^{-1}(Y_0) = \fX_0 \cap G$. If $\fX_t$ is a graded subscheme of $\Spec\ \!\C[t,\mathbf{z}]$ with respect to the grading induced by that on $\C[\mathbf{z}]$ resulting from $\deg(z_{ij})= e_j,$ then $[Y]=[Y_0]$. That is, the cohomology classes of $Y$ and $Y_0$ in $H^*(G/B)$ are equal.
\end{rem}

\subsection{Miracle flatness and applications}

\emph{Miracle flatness} refers to a specific context in which a family of algebraic varieties is necessarily flat. We shall see that miracle flatness cannot be used to prove Theorem~\ref{thm.flat}; see Remark~\ref{rem.notequidim}. However, since miracle flatness is the main tool used in the literature for proving that certain families of Hessenberg varieties are flat, it makes sense to address it here.
We highlight two cases in which it applies, and provide clear illustration for the failure of miracle flatness for the family of Hessenberg schemes over the minimal sheet.

The following version of miracle flatness can be found in~\cite[\S 23, Thm.~23.1]{Matsumura}.  Roughly speaking, this result allows us to conclude a morphism is flat provided that the fibers are all the same dimension. 

\begin{thm}[Miracle Flatness] \label{thm.miracle} Suppose $\cx$ and $\fX$ are irreducible schemes and $\varphi: \cx \to \fX$ a morphism. Suppose $\fX$ is regular, $\cx$ is Cohen--Macaulay, $\varphi$ takes closed points of $\cx$ to closed points of $\fX$, and for every closed point $y\in \fX$, the fiber $\varphi^{-1}(y)$ has dimension $\dim \cx -\dim \fX$ \textup{(}or is empty\textup{)}. Then $\varphi$ is flat.
\end{thm}

Flatness for regular Hessenberg varieties was originally proved for $\mx$ semisimple in \cite{ADGH2018} and in full generality by Abe, Fujita, and Zeng in~\cite[\S 6.2]{Abe-Fujita-Zeng2020}.  Hessenberg varieties associated to the Hessenberg function $h=(n-1, n,\ldots, n)$ are studied in detail by the second author, Escobar, and Shareshian in~\cite{EPS-online}. In both cases, miracle flatness can be used to deduce that the family $\varphi_{\mx,h}$ is flat. The proof of the following propositions may be found in the appendix (Section~\ref{sec.miracleflatnessproofs}).

\begin{prop}\label{prop.miracle} Given $\mx\in \mathfrak{gl}_n(\C)$ and $h:[n]\to[n]$ a Hessenberg function, let $\varphi_{\mx, h}: \fX_{\mx_t, h}\to \Spec\ \!\C[t]$ be the morphism~\eqref{eqn.morphism}.  Suppose that for each $a\in \C$, 
\begin{itemize}
    \item $\fX_{\mx_a,h}$ is irreducible, and 
    \item $\dim \fX_{\mx_a,h}=\sum_{i=1}^n h(i)$.
\end{itemize} 
Then $\varphi_{\mx, h}$ is flat.
\end{prop}

This dimension criterion may be used to identify several cases in which there is a flat family of matrix Hessenberg schemes:
\begin{prop} Suppose $\mx\in \mathfrak{gl}_n(\C)$ and $h: [n]\to [n]$ is a Hessenberg function.  Assume further that either
\begin{enumerate}
\item $\mx$ is an element of the regular sheet $\fg_{reg}$ and $h(i)>i$, or that
\item $h = (n-1, n, \ldots, n)$.
\end{enumerate}
Then $\varphi_{\mx, h}$ is flat.
\end{prop}

\begin{rem}\label{rem.notequidim}
When $\mx$ is in the minimal sheet, Theorem~\ref{thm.miracle} does not apply to $\varphi_{\mx, h}$.  Any Cohen--Macaulay scheme must be equidimensional~\cite[Cor.~18.11]{Eisenbud}, yet Example~\ref{ex.flat.h2444} shows that this is not always the case over the minimal sheet.
\end{rem} 

The identification of two cases in which miracle flatness applies, and one case in which it does not, leads to the following open question.

\begin{Question}
    For which $\mx$ and $h$ does miracle flatness apply to the family $\varphi_{\mx, h}: \fX_{\mx_t, h}\rightarrow \Spec \ \! \C[t]$? More specifically, in which cases does Proposition~\ref{prop.miracle} apply?
\end{Question}


\section{Technical lemmas}\label{sec.technical}

This section contains the technical heart of the paper. We introduce two collections of ideals in the ring $\C[\mathbf{z}, t]$ and study their properties.  


\subsection{The ideals $\cjt_i$ and $\ck_j$}  For $i,j\in [n]$, let 
\begin{equation}
  \begin{aligned}\label{eq.JtandK}
\cjt_i&:= \left< tz_{11}+z_{n1}, \ldots, t z_{1i}+z_{ni} \right>, \mbox{ and}\\
\ck_{j} & := \langle  p_B\mid B\subseteq \{2,\ldots, n\}, |B|= j \rangle 
\end{aligned}
\end{equation}
where $p_B$ is the determinant of the submatrix $(z_{k\ell})_{k\in B, 1\leq \ell \leq j}$ of $Z=(z_{ij})$, and set $\cj_0^t = \langle 0\rangle$. 
Recall that $\ck_j$ was already introduced in Section~\ref{sec:matrixSchubs}, where we note that $\ck_n=\langle 0 \rangle$.

Our motivation for studying these ideals comes from the fact that the ideal $\ci_{\semi_t, h}$ can be written in terms of the $\cjt_i$ and $\ck_i$'s for $\semi_t$ the sheet line in $\fg_{(2,1^{n-2})}$ from~\eqref{eqn.xt.min}.

\begin{lemma}\label{lemma.generators} Let $\semi=\diag(1,0,\dots, 0)$ and $\semi_t\in M_n(\C[t])$ be the sheet line in the minimal sheet as in~\eqref{eqn.xt.min}, 
$$
\semi_t = t\semi+E_{1n}.
$$
For any Hessenberg function $h$,
\[
\ci_{\semi_t, h} := \sum_{i \in \cc(h)} \cjt_{i^*} \cdot \ck_{h(i)}.
\]
\end{lemma}

\begin{proof} We argue that $\ci_{\semi_t, h, i} = \cjt_{i} \cdot \ck_{h(i)}$ for all $i\in [n]$.  The desired results then follow from Lemma~\ref{lem.corners}, using the fact that $h(i)=h(i^*)$ by definition.

Suppose first that $h(i)=n$. Then $\ck_{h(i)} = \ck_n=0$ and recall that $\ci_{\mx_t, h, i}$ is defined to be the zero ideal in this case.  Thus $\ci_{\semi_t, h, i} = \cjt_i \cdot \ck_{h(i)}$ holds trivially in this case.

Suppose now that $h(i)<n$.  By definition, $\ci_{\semi_t, h, i}$ is generated by the $(h(i)+1)\times (h(i)+1)$-minors of the matrix
\[
\begin{bmatrix} tz_{11}+z_{n1} & tz_{12}+z_{n2} & \cdots & tz_{1i}+z_{ni} & z_{11} & z_{12} & \cdots  & z_{1h(i)} \\ 
 0 & 0 & \cdots & 0 & z_{21} & z_{22} & \cdots  & z_{2h(i)} \\ 
\vdots & \vdots & \ddots & \vdots & \vdots & \vdots & \ddots & \vdots \\
 0 & 0 & \cdots & 0 & z_{n1} & z_{n2} & \cdots  & z_{nh(i)}
\end{bmatrix}.
\]
Selecting two or more of the first $i$ columns will result in the corresponding $(h(i)+1)\times (h(i)+1)$-minor being equal to zero.  Thus, we may assume that each generator of $\ci_{\semi_t, h, i}$ is an  $(h(i)+1)\times (h(i)+1)$-minor of 
\[
\begin{bmatrix} tz_{1j}+z_{nj} & z_{11} & z_{12} & \cdots & z_{1h(i)}\\ 0 & z_{21} & z_{22} & \cdots & z_{2h(i)} \\ 
\vdots & \vdots & \vdots & \ddots & \vdots\\
0 & z_{n1} & z_{n2} & \cdots & v_{n h(i)}
 \end{bmatrix}
\]
for some $1\leq j \leq i$, each of which is of the form $(tz_{1j}+z_{nj})p_B$ for some $B\subseteq \{2,\ldots, n\}$ such that $|B|=h(i)$.
\end{proof}

For each $a\in \C$ let
\begin{eqnarray}\label{eqn.evaluationhom}
\ev_a:\C[\mathbf{z}, t] \to \C[\mathbf{z}]
\end{eqnarray}
denote the evaluation homomorphism such that $\ev_a(t)=a$. Recall that  Lemma~\ref{lemma.fibers} shows that the scheme theoretic fiber over $\left< t-a\right>$ of $\varphi_{\mx, h}$ is isomorphic to the matrix Hessenberg scheme defined by the ideal $\ev_a(\tilde{\ci}_{\mx_t, h}) = \tilde{\ci}_{\mx_a, h}$.  
For each $i\in [n-1]$ we write $\cj_i^a:= \ev_a(\cjt_i)$. Note that this notation aligns with that of $\cj_i^0$ defined in Section~\ref{sec:matrixSchubs}, namely $\cj_i^0 = \ev_0(\cjt_i)$.

\begin{lemma}\label{lem.ev2} For all $a\in \C^*$, let $\psi_a:\C[\mathbf{z}] \to \C[\mathbf{z}]$ be the ring homorphism such that $\psi_a(z_{1\ell}) = a^{-1}(z_{1\ell}-z_{n\ell})$ for all $\ell\in [n]$ and $\psi_a(z_{k\ell}) = z_{k\ell}$ for all $k>1$ and $\ell\in [n]$.  Then for all $i,j\in [n]$ with $i\leq j,$
\begin{enumerate}
\item $\psi_a(\cj_{i-1}^a) = w_0\cdot \cj_{i-1}^0$ is the Schubert determinantal ideal corresponding to the permutation $u[i]$, and
\item $\psi_a(\cj_{i-1}^a + \ck_j) = w_0\cdot\cj_{i-1}^0 + \ck_{j}$ is the prime ideal defining the matrix Richardson variety $\cx_{u[i]}^\op\cap \cx_{w_0v[j]}$.
\end{enumerate}
\end{lemma}

\begin{proof}The first statement follows immediately from Lemma~\ref{lem.Schubert}.  The second is a direct result of combining statement (1) with Lemma~\ref{lem.Schubert4}.
\end{proof}

\begin{lemma}\label{lemma.inclusions} The following statements all hold for $i,j,k,\ell\in [n-1]$.
\begin{enumerate}
\item If $i\leq j$, then $\cjt_i \subseteq \cjt_j$ and $\ck_{j}\subseteq \ck_i$.
\item Suppose $i<j$ and $k<\ell$.  We have $\cjt_i+ \ck_j \subseteq \cjt_k+ \ck_\ell$ if and only if $i\leq k$ and $j \geq \ell$.
\end{enumerate}
\end{lemma}
\begin{proof} The proof of this lemma when $t=0$ is found in~\cite[Lemma 2.4]{GP-semisimple}, and the more general proof is identical. 
\end{proof}


\subsection{Gr\"obner background} We recall essential facts about Gr\"obner geometry needed for the technical arguments of the next section.
Let $R$ be a polynomial ring over an algebraically closed field $k$.  Let $\preceq$ be a fixed monomial ordering on $R$. Denote the leading term of a nonzero polynomial $f\in R$ by $\LT(f)$ and let $\LM(f)$ be the monic monomial such that $\LT(f) = c\, \LM(f)$ for some $c\in k$, called the leading monomial of~$f$.

Given a nonzero ideal $\ci$ we let $\LT_{\preceq}(\ci)$ be the ideal generated by the leading terms of the elements of $\ci$.
A finite subset $G=\{g_1, \ldots, g_r\}$ of a nonzero ideal $\ci \subseteq R$ is a \emph{Gr\"obner basis} if 
\[
\LT_{\preceq}(\ci) = \left< \LT(g_1), \ldots, \LT(g_r) \right>.
\]
By convention, the empty set is a Gr\"obner basis of the zero idea. 

Given two polynomials $f_1, f_2\in R$, we let $\M(f_1, f_2)$ denote the (monic) least common multiple of $\LM(f_1)$ and $\LM(f_2)$.  The \emph{$S$-polynomial} of $f_1$ and $f_2$ is 
\begin{eqnarray}\label{eqn.S-poly.def}
S(f_1, f_2) := \frac{\M(f_1,f_2)}{\LT(f_1)} f_1 - \frac{\M(f_1, f_2)}{\LT(f_2)}f_2.
\end{eqnarray}
Buchberger's Criterion states that $G=\{g_1, \ldots, g_r\}$ is a Gr\"ober basis of $\ci= \left< g_1, \ldots, g_r \right>$ if and only if for all pairs $i\neq j$ the remainder of $S(g_i, g_j)$ on division by $G$ (with respect to some fixed order) is zero.  We make use of a refinement of Buchberger's Criterion; to state the refinement, we need the following definition.

Let $G = \{g_1, \ldots, g_r\}\subseteq R$.  Given $f\in R$ we say $f$ \emph{reduces to zero modulo $G$}, denoted by $f \rightarrow_G 0$, if $f$ has a \emph{standard representation}:
\[
f = a_1g_1 + \cdots +a_rg_r, \; a_1, \ldots, a_r \in R
\]
such that $a_ig_i\neq 0$ implies $\LM(a_ig_i) \preceq \LM(f)$ for all $i$.  The following is~\cite[Chap.~2, \S9, Thm.~3]{CLO}.

\begin{thm}\label{thm.Grobner} A generating set $G=\{g_1, \ldots, g_r\}$ of an ideal $\ci$ in $R$ is a Gr\"obner basis if and only if $S(g_i,g_j) \rightarrow_G 0$ for all $i\neq j$.
\end{thm}

We use the following lemmas to simplify our arguments below; the following is \cite[Chap.~2, \S9, Prop.~4]{CLO}.

\begin{lemma}\label{lemma.rel.prime} Let $G\subseteq R$ be a finite set of polynomials and suppose $f , g \in G$ such that the leading monomials of $f$ and $g$ are relatively prime.  Then $S(f,g) \rightarrow_G 0$. 
\end{lemma}

For finite sets $G, H \subseteq R$, let
$$
GH := \{gh: g\in G, h\in H\}
$$
denote the set of products of polynomials $g\in G$ and $h\in H$. 

\begin{lemma}\label{lemma.Spoly} Suppose $G, H\subseteq R$ are two finite sets of polynomials.
Let $g_1, g_2\in G$ and $h_1,h_2\in H$ with the property that the leading monomials of $g_1$ and $g_2$ are each relatively prime to the leading monomials of both $h_1$ and $h_2$.  If $S(g_1, g_2)\rightarrow_G 0$ and $S(h_1, h_2)\rightarrow_H 0$, then $S(g_1h_1, g_2h_2)\rightarrow_{GH} 0$.
\end{lemma}

\begin{proof}
    We calculate the $S$ polynomial directly. Since $\LM(gh) = \LM(g) \LM(h)$ for all $g\in G, h\in H$,  (see~\cite[Chap.~2, \S2, Lemma 8]{CLO}), our assumptions imply $\M(g_1h_1,g_2h_2) = \M(g_1, g_2)\M(h_1,h_2)$.
For the rest of this proof, let $\M_g=\M(g_1, g_2)$ and $\M_h =\M(h_1,h_2)$. A straightforward calculation shows
\reqnomode
\begin{align}\label{eq.Sprod}
S(g_1h_1,g_2h_2) 
&=\frac{\M_g
\M_h
}{ \LT(g_1) \LT(h_1)} g_1 h_1 -  \frac{\M_g\M_h}{ \LT(g_2) \LT(h_2)} g_2 h_2 \nonumber \\
&=\frac{\M_g\M_h }{ \LT(g_1) \LT(h_1)} g_1 h_1 -  \frac{\M_g\M_h  }{ \LT(g_1) \LT(h_2)} g_1 h_2 \nonumber\\
&\phantom{=\frac{\M_g\M_h }{ \LT(g_1) \LT(h_1)} g_1 h_1\ }
+ \frac{\M_g\M_h }{ \LT(g_1) \LT(h_2)} g_1 h_2   -  \frac{\M_g\M_h}{ \LT(g_2) \LT(h_2)} g_2 h_2 \nonumber\\
&= \frac{\M_g}{\LT(g_1)} g_1 \left( \frac{\M_h}{\LT(h_1)} h_1 -  \frac{\M_h}{\LT(h_2)} h_2 \right)
+ \frac{\M_h}{ \LT(h_2)} h_2 \left( \frac{\M_g}{\LT(g_1)} g_1 - \frac{\M_g}{\LT(g_2)}g_2 \right)\nonumber\\
&=  \frac{\M_g}{\LT(g_1)} g_1 S(h_1, h_2) + \frac{\M_h}{\LT(h_2)} h_2 S(g_1, g_2).
\end{align}
If $S(g_1h_1,g_2h_2)\neq 0$ then together with~\cite[Chap.~2, \S2 Lemma 8]{CLO}, this equation implies
\leqnomode
\begin{equation}\label{eq:LMmax}
    \LM(S(g_1h_1,g_2h_2)) \preceq \max\left\{ \M_g\LM(S(h_1,h_2)), \M_h\LM(S(g_1,g_2)) \right\}
\end{equation}
with equality if and only if $\M_g\LM(S(h_1, h_2))\neq \M_h\LM(S(g_1, g_2))$. 
Indeed, suppose $\M_g\LM(S(h_1, h_2))= \M_h\LM(S(g_1, g_2))$.  Since $\LM(S(g_1, g_2)) \prec \M_g$ by 
construction of the $S$-polynomial (see~\cite[Chap.~3, \S6, Exercise 7]{CLO}), some factor of $\M_g$ must divide $\M_h$.  This in turn implies that at least one of $\LM(g_1)$ or $\LM(g_2)$ must share a common factor with either $\LM(h_1)$ or $\LM(h_2)$, contrary to the assumption that these leading monomials are relatively prime. 
Thus equality holds in \eqref{eq:LMmax}.
Substituting the standard representation for $S(h_1,h_2)$ with respect to $H$ and the standard representation of $S(g_1,g_2)$ with respect to $G$ into the formula~\eqref{eq.Sprod} 
yields a standard representation for $S(g_1h_1, g_2h_2)$  with respect to $GH$. 
\end{proof}

We also use the following key property of Gr\"obner bases; see~\cite[Chap.~3, \S1, Thm.~2]{CLO}. 

\begin{prop}[The Elimination Property] \label{prop.elimination} Let $\ci \subseteq R=k[x_1, x_2, \ldots, x_n]$ be an ideal and let $G=\{g_1, \ldots, g_r\}$ be a Gr\"obner basis of $\ci$ with respect to the lexicographic monomial order such that $x_1 \succ x_2 \succ \cdots \succ x_n$.  Then for each $0\leq \ell \leq n$ the set $G\cap k[x_{\ell+1}, \ldots, x_n]$ is a Gr\"obner basis of the ideal $\ci \cap k[x_{\ell+1}, \ldots, n]$.
\end{prop}


\subsection{Saturation proofs}

Consider the localization of $\C[t,\mathbf{z}]$ with respect to the multiplicative set $S=\{1, t, t^2, \ldots \}$, denoted $S^{-1}\C[t,\mathbf{z}]$, together with the canonical homomorphism $\kappa = \pi\circ\iota$ given by the composition of inclusion followed by the natural projection:
\begin{eqnarray}\label{eqn.localization}
\begin{tikzcd}
    \C[t, \mathbf{z}] \arrow[hook]{r}{\iota} &\C[s, t, \mathbf{z}] \arrow{r}{\pi}&\C[s, t,\mathbf{z}]/ \left< st-1\right> \cong S^{-1}\C[t,\mathbf{z}.]
\end{tikzcd}
\end{eqnarray}
The ideal in $S^{-1}\C[t,\mathbf{z}]$ generated by the image of an ideal $\ci$ in $\C[t,\bf{z}]$ is called its {\em localization}, denoted $S^{-1}\ci$.   Every ideal in $S^{-1}\C[t,z]$ is $S^{-1}\ci$ for some ideal $\ci$  in $\C[t,z]$ (though not uniquely). 

The \emph{saturation} of $\ci$ with respect to the element $t$ is the preimage $\kappa^{-1}(S^{-1}\ci)$, often denoted $(\ci:t^\infty).$ Alternatively, $(\ci:t^\infty)$ is the ideal generated by all $f\in \C[t,\mathbf{z}]$ such that $t^k f\in \ci$ for some $k$. 
We say that $\ci$ is \emph{$t$-saturated} if $\ci =(\ci:t^\infty)$. Observe that ideal $\ci$ is $t$-saturated if and only if 
\begin{align*}
    \ci &= \iota^{-1}\circ \pi^{-1}(S^{-1}\ci) = 
\left(\iota(\ci)
+ \langle st-1\rangle\right)\cap \C[t,z]
\end{align*}
where, by abuse of notation, $\iota(\ci)$ indicates the ideal in $\C[s,t,z]$ generated by $\iota(\ci).$

Our main result in this section is the following.

\begin{prop}\label{prop.saturated} Let $\cj_i^t$ and $\ck_j$ be as in Equation~\eqref{eq.JtandK}. The ideal 
$\cj_i^t + \ck_j$ is $t$-saturated for all $0\leq i<j \leq n$,
\end{prop}

We note that
$\cj_i^t + \ck_j$ is $t$-saturated if and only if 
\begin{eqnarray}\label{eqn.saturated}
(\iota(\cj_i^t+\ck_j)
+ \langle st-1 \rangle) 
\cap \C[t, \mathbf{z}] = \cj_i^t + \ck_j,
\end{eqnarray}
where $\iota(\cj_i^t + \ck_j)$ indicates the ideal generated by $\cj_i^t + \ck_j$ in $\C[s,t,\mathbf{z}].$

 The outline of the proof is as follows. We identify a Gr\"obner basis $H$ of 
$\iota(\cj_i^t+\ck_j)
+\langle ts-1\rangle$ in
$\C[s,t,\mathbf{z}],$  such that 
$$
G:= H\cap \C[t,\mathbf{z}]
$$ generates  $\cj_i^t+\ck_j$ in $\C[t,\mathbf{z}]$. We then apply the Elimination  Property (Proposition~\ref{prop.elimination}), to conclude that $G$ is a Gr\"obner basis of the intersection $\left(\iota(\cj_i^t+\ck_j)+\langle ts-1\rangle\right)\cap \C[t,\mathbf{z}]$. But since $G$ generates $\cj_i^t+\ck_j$, the two ideals are equal, implying that $\cj_i^t+\ck_j$  is $t$-saturated.

Let $\preceq$ denote the lexicographic monomial order on $\C[s,t,\mathbf{z}]$ induced by 
\begin{equation}\label{eqn.order}
    \begin{aligned}
        s\succ t &\succ z_{1n}  \succ z_{1\,n-1} \succ \cdots \succ z_{12} \succ z_{11} \\
        &\succ z_{21} \succ z_{22} \succ \cdots \succ z_{2n} \succ      
        \cdots \succ z_{n1} \succ z_{n2} \succ \cdots \succ z_{nn},
    \end{aligned}
\end{equation}
where we note the decreasing second index when the first index is 1, and the increasing second index subsequently. 
Using this ordering, for $B\subseteq \{2,\dots, n\}$ with $|B|=j$, 
\begin{equation}\label{eqn.LTpB} \LT (p_B) = z_{b_1 1}z_{b_2 2} \cdots z_{b_j j},  \textup{ where } B=\{b_1<b_2<\cdots<b_j\}.
\end{equation}

In particular, $\preceq$ yields a \emph{diagonal term order} on the minors used to define $\ck_j$, since the variable $z_{1k}$ are not factors of any term of $p_B$ for any $B \subset \{2,\dots, n\}$. That is, the leading monomial of each minor defining $\ck_j$ is  
the product of the diagonal entries of the corresponding submatrix.

\begin{lemma}\label{lemma.Grobner2}  Consider the  polynomials
$$
g_k := tz_{1k}+z_{nk},\quad h_{k\ell} := z_{1\ell}z_{nk} - z_{1k}z_{n\ell},\quad f_k := sz_{nk} + z_{1k}.
$$
Define sets
\begin{align*}
    H_i&:= \{g_k \mid 1\leq k\leq i\} 
    \cup \{h_{k\ell}\mid 1\leq k< \ell \leq i]\} \cup  \left\{  f_k \mid 1\leq k \leq i  \right\}\cup \left\{  st-1 \right\} \\
    H^j &:= \{p_B\mid B\subseteq \{2, \dots, n\}, |B|=j\} \cup  \left\{  st-1 \right\}.
\end{align*}
 Then for $i\in [n]$ and $j\in [n-1],$ 
\begin{enumerate}
\item  
The set $H_i$ is a Gr\"obner basis for  $\iota(\cj_i^t)+ \left< st-1 \right> \subseteq \C[s,t,\mathbf{z}]$.
\item The set $H^j$
is a Gr\"obner basis for  $\iota(\ck_j)+ \left< st-1 \right> \subseteq \C[s,t,\mathbf{z}]$.
\item The set $H := H_i \cup H^j$
is a Gr\"obner basis for $\iota(\cj_i^t+\ck_j) + \left< st-1 \right> \subseteq \C[s,t,\mathbf{z}]$ when $i<j$. 
\end{enumerate}
\end{lemma}

\begin{proof} We first show that the elements in each set generate the corresponding ideals. Since $st-1\in H_i$ and
$\cj_i^t = \langle g_1, \dots, g_i\rangle,$  the set $H_i$ clearly generates an ideal that contains $\iota(\cj_i^t)+\langle st-1\rangle$. On the other hand,  $h_{k\ell} = z_{1\ell}g_k - z_{1k}g_\ell \in \iota(\cj_i^t)$ for $k< \ell\leq i$ and $f_k = sg_k - z_{1k}(st-1)\in\iota(\cj_i^t)+\langle st-1\rangle$ for $k\leq i$, so $H_i$ generates no more than $\iota(\cj_i^t)+\langle st-1\rangle$. It is immediate that $H^j$ generates the ideal $\iota(\ck_j)+\langle st-1\rangle$. The union therefore generates the sum of the ideals.

We show that each pair of $S$-polynomial reduces to $0$ modulo the corresponding set. Consider the following chart of leading terms using the monomial ordering from above:
\begin{center}
\begin{tabular}{c|c|c}
Polynomial & Leading Term & Index Set
\\
\hline
$st-1$ & $st$ &
\\
\hline
\multirow{2}{*}{$p_B$} & \multirow{2}{*}{$z_{b_1 1}z_{b_2 2} \cdots z_{b_j j}$} & $B=\{b_1<b_2<\cdots<b_j\}$ \\
&  &$B\subset\{2,\dots, n\}$, $1\leq j\leq n-1$\\
\hline
$g_k$ & $tz_{1k}$
& $1\leq k\leq n$\\
\hline 
$h_{k\ell}$ & $z_{1\ell}z_{nk}$ & $1\leq k<\ell\leq n$\\
\hline
$f_k$ & $sz_{nk}$ & $1\leq k\leq n$
\end{tabular}
\end{center}

We observe by inspection that many of these polynomials have leading monomials with no common factors, which implies that their corresponding $S$-polynomials are $0$ modulo the corresponding set by Lemma~\ref{lemma.rel.prime}. We then prove that the remaining $S$-polynomials have a standard representation, showing that the set forms a Gr\"obner basis using Theorem~\ref{thm.Grobner}. 

\vspace{.1in}
\noindent Proof of (1). The polynomials in $H_i$ have the following properties:
\begin{itemize}
\item The leading term of $st-1$ is relatively prime to the leading term of $h_{k\ell}$ for all index choices. 
\item The leading term of $g_k$ is relatively prime to the leading terms of $h_{ab}$ for $k\neq b$ and $f_\ell$ for all $\ell$. 
\item The leading terms of $h_{k\ell}$ and $h_{ab}$ are relatively prime when both $k\neq a$ and $\ell\neq b$.
\item The leading terms of $h_{k\ell}$ and $f_a$ are relatively prime if $a\neq k$.
\end{itemize}

Thus by Lemma~\ref{lemma.rel.prime}, the $S$-polynomials vanish for these pairs of polynomials, modulo any set that contains them both, so in particular, modulo $H_i$. 

The remaining pairs of functions in $H_i$ have $S$-polynomials tabulated here:
\begin{center}
\begin{tabular}{c|c|c|c|c}
$f$ & $g$ & index set &  $\M(f,g)$ &   $S(f,g)$ \\ \hline
$g_k$ & $g_\ell$ & $1\leq k<\ell \leq i$ & $tz_{1k}z_{1\ell}$ & $h_{k\ell}$ \\ \hline
$g_k$ & $h_{ak}$ & $1\leq a<k \leq i$ & $tz_{1k}z_{na}$ &  $ z_{nk}g_{a}$ \\ 
\hline
$h_{k\ell}$ & $h_{kb}$ & $1\leq k<b<\ell \leq i$ & $z_{1\ell} z_{1b}z_{nk}$ & $z_{1k}h_{b\ell}$\\
\hline
$h_{k\ell}$ & $h_{a\ell}$ & $1\leq k<a< \ell \leq i$ &$z_{1\ell} z_{nk}z_{na}$ & $z_{n\ell} h_{ka}$\\
\hline
$g_k$ & $st-1$ & $1 \leq k \leq i$ &  $stz_{1k}$ &  $f_k$ \\ \hline
$h_{k\ell}$ & $f_k$ &  $1 \leq k<\ell \leq i,$ & $sz_{1\ell}z_{nk}$ & $-z_{1k} f_\ell$ \\
\hline
$f_k$ & $f_\ell$ &$1\leq k\leq \ell \leq i$& $sz_{1\ell} z_{nk}$& $-h_{k\ell}$\\
\hline
$f_k$ & $st-1$ & $1 \leq k\leq i$& $stz_{nk}$& $g_k$  
\end{tabular}
\end{center}
In each case, the resulting $S$-polynomial clearly has a standard representation and thus reduces to zero modulo $H_i$.
Thus we have shown that $H_i$ is a Gr\"obner basis of $\iota(\cj_i^t)+\langle st-1\rangle$.

\vspace{.1in}
\noindent Proof of (2). Now consider polynomials in $H^j.$ Since the leading term of $st-1$ is relatively prime to the leading terms of $p_B$ for all index choices, $S(p_B, st-1)\rightarrow_{H^j} 0$. 

It is well known (see e.g.~\cite[Thm.~16.28]{MS05}) that the set $\{w_0\cdot p_B \mid B\subseteq \{2,\ldots, n\}, |B|=j\}$ is a Gr\"obner basis for~$w_0\cdot \ck_j$ with respect to any monomial order for which the leading term of each minor $p_B$ is a product of anti-diagonal entries.  The action of $w_0$ sends anti-diagonal terms to diagonal ones and~\eqref{eqn.order} is a diagonal term order. Thus the set $\{p_B\mid B\subseteq \{1,\ldots, n\}, |B|=j\}$ is a Gr\"obner basis for $\ck_j$. In particular, $S(p_A, p_B)\rightarrow_{H^j}0$ for all $A, B\subset \{2, \dots, n\}$ with $|A|=|B|=j$ by Theorem~\ref{thm.Grobner}. The desired result now follows from a second application of Theorem~\ref{thm.Grobner}.

\vspace{.1in}

\noindent Proof of (3).
Following the first two parts of the proposition, we need only check the $S$-polynomials for pairs of functions $f\in H_i$ and $g\in H^j$. We may ignore pairs including $st-1$ as it occurs in both sets. Again considering the leading terms from the table, we readily observe: 
\begin{itemize}
\item The leading term of $p_B$ is relatively prime to the leading term of $g_k$ for all $k$.
\item The leading term of $p_B$ is relatively prime to the leading terms of $h_{k\ell}$ and $f_k$, unless $b_j=n$ and $j=k$.  By assumption $k\leq i < j$, so the polynomials are relatively prime in all cases. 
\end{itemize}
It now follows that $S(f,g)\rightarrow_H 0$ for all $f,g\in H$ by Lemma~\ref{lemma.rel.prime}. Since $H$ generates the ideal $\iota(\cj_i^t)+\ck_j$, we conclude that $H$ is a Gr\"obner basis of this ideal by Theorem~\ref{thm.Grobner}. \end{proof}

We may now prove the main result. 
\begin{proof}[Proof of Proposition~\ref{prop.saturated}]
Since $s\succ t$, we may apply Proposition~\ref{prop.elimination} to each of $H_i, H^j$ and $H$. 
For $1\leq i<j\leq n-1,$ we conclude
$$
H\cap \C[t,\mathbf{z}]= \{g_k: k\in [i]\}\cup \{h_{k\ell}: k,\ell\in [i], k<\ell\}\cup \{p_B: B\subset \{2, \dots, n\}, |B|=j\}
$$ is a Gr\"obner basis of $$
(\iota(\cj_i^t+\ck_j)+ \langle st-1\rangle) \cap \C[t,\mathbf{z}]. 
$$
Since $\cj_i^t = \langle g_1, \dots, g_i\rangle$ and $\ck_j = \langle p_B \mid  B\subset \{2, \dots, n\}, |B|=j \rangle$, it is clear that $H\cap \C[t,\mathbf{z}]$ generates an ideal that contains $\cj_i^t + \ck_j$. On the other hand, we showed that $h_{k\ell}\in \cj_i^t$, so $H\cap \C[t,\mathbf{z}]$ generates and ideal contained in $\cj_i^t + \ck_j$. It follows that 
$$
(\iota(\cj_i^t+\ck_j)+ \langle st-1\rangle) \cap \C[t,\mathbf{z}] = \cj_i^t+\ck_j,
$$
which implies the sum of ideals is $t$-saturated. The same argument applied to $H_i$ shows that $H_i\cap \C[t,\mathbf{z}]$ is a Gr\"obner basis of $\cj_i^t$ and that this ideal is $t$-saturated. This proves the case that $j=n$, when $\cj_i^t+\ck_j = \cj_i^t$. Finally, applying the same argument to $H^j$ shows that $H^j\cap \C[t,\mathbf{z}]$ is a Gr\"obner basis of $\ck_j$, and implying that this ideal is $t$-saturated. This observation covers the case that $i=0$, when $\cj_i^t+\ck_j = \ck_j$. \end{proof}

We record two corollaries of the proof above for use later.

\begin{cor} \label{cor.sq.free} 
Let $B\subseteq \{2, \ldots, n\}$ with $|B|=j\leq n-1$ and $i\leq j$. Then
\begin{enumerate}
\item The leading monomials of $g_k$ and $p_B$ are relatively prime for $1\leq k \leq i$, and
\item The leading monomials of $h_{k\ell}$ and $p_B$ are relatively prime for  $1\leq k < \ell \leq i$.
\end{enumerate}
In particular, both $\LT(g_kp_B)$ and $\LT(h_{k\ell}p_B)$ square-free monomials.
\end{cor}

\begin{cor} \label{cor.Grobner} The set $G_i:= H_i\cap \C[t,\mathbf{z}]$ is a Gr\"obner basis of $\cjt_i$ for all $i\in [n]$.  The set $G^j:= H^j\cap \C[t,\mathbf{z}]$ is a Gr\"obner basis of $\ck_j$ for all $j\in [n-1]$.
\end{cor}

\begin{proof} Statement (1) is obvious and statement (2) follows by the proof of Lemma~\ref{lemma.Grobner2}(3) above. The last assertion of the statement follows immediately since $\LT(g_k)$, $\LT(h_{k\ell})$, and $\LT(p_B)$ are all square-free monomials.
\end{proof}

The next example shows that the assumption $i<j$ in Lemma~\ref{lemma.Grobner2}(3) is necessary.

\begin{example} Let $n=4$ and consider the ideals $\cj_2^t = \left< g_1, g_2 \right>$ and $\ck_2 = \left< p_{\{2,3\}}, p_{\{2,4\}}, p_{\{3,4\}} \right>$ where $g_1 =  tz_{11}+z_{41}$ and $g_2 = tz_{12}+z_{42}$ and 
\[
p_{\{2,3\}} = z_{21}z_{32}-z_{22}z_{31},\; p_{\{2,4\}} = z_{21}z_{42}-z_{22}z_{41}, \textup{ and }  p_{\{3,4\}} = z_{31}z_{42}-z_{32}z_{41}.
\]
Both $\cj_2^t$ and $\ck_2$ are $t$-saturated by Proposition~\ref{prop.saturated}.  However, the indices $i=j=2$ of the ideals do not satisfy the hypotheses of Proposition~\ref{prop.saturated} or Lemma~\ref{lemma.Grobner2}(3) for the sum.
Indeed, consider the elements $f_2:=sz_{42}+z_{12}$ and $p_{\{2,4\}} $ in $H$.
The leading monomials of $f_2$ and $p_{\{2,4\}}$ are not relatively prime and we have  
\[
S(f_2, p_{\{2,4\}}) = \frac{sz_{21}z_{42}}{sz_{42}} f_2 - \frac{sz_{21}z_{42}}{z_{21}z_{42}}  p_{\{2,4\}} = z_{21}z_{12} + sz_{22}z_{41}.
\]
Recall that $f_1 = sz_{41}+z_{11} \in H$.  
We may write 
\[
S(f_2, p_{\{2,4\}}) = z_{22} f_1 + z_{21}z_{12}-z_{22}z_{11}
\] 
but since $z_{21}z_{12}-z_{22}z_{11}$ is not divisible by the elements of $H$ 
in any order, we conclude that $S(f_2, p_{\{2,4\}})$ does not reduce to zero modulo $H$.
The fact that $S(f_2, p_{\{2,4\}})$ has a non-zero remainder modulo $H$ that lies in $\C[t,\mathbf{z}]$ tells us the ideal $\cj_2^t + \ck_2$ is not $t$-saturated.  Indeed,
\[
t(z_{21}z_{12}-z_{22}z_{11}) = z_{21}g_2  - z_{22} g_1 + p_{\{2,4\}} \in \cj_2^t+\ck_2
\]
but $z_{21}z_{12}-z_{22}z_{11} \notin \cj_2^t+\ck_2$.
\end{example}


\subsection{Prime and radical ideal proofs}

We now get the following.

\begin{thm}\label{thm.prime} 
The ideal $\cj_i^t + \ck_j \subset \C[t,\mathbf{z}]$ is prime for all  $i, j$ with $0 \leq i<j \leq n$.
\end{thm}
\begin{proof} 
Suppose first that $i=0$ so $\cj_i^t + \ck_j = \ck_j$.  When $j=n$,
$\ck_n=\left< 0 \right>$ and the claim holds trivially.  When $j<n$, Lemma~\ref{lem.Schubert} implies $\ck_j$ is a Schubert determinantal ideal, and is therefore prime.

Next, we suppose $j=n$ so $\cj_i^t+\ck_j = \cj_i^t$.
Recall that the localization map~\eqref{eqn.localization} defines a bijection between saturated prime ideals $\cp\subset \C[t,\mathbf{z}]$ and prime ideals in the localization ring $S^{-1}\C[t,\mathbf{z}]$, defined by $\cp \mapsto S^{-1}\cp$. By Proposition~\ref{prop.saturated}, $\cj_i^t$ is $t$-saturated, so it suffices to show that $S^{-1}\cj_i^t$ is prime in $S^{-1}\C[t, \mathbf{z}]$.
Consider the ring isomorphism
\[
\varphi: S^{-1}\C[t,\mathbf{z}] \to \C(t) [\mathbf{z}]
\]
defined by $\varphi(z_{1j}) = \frac{1}{t}(z_{1j}-z_{nj})$ for $j\in [n]$ and the identity on all other variables, extended algebraically. Now 
\[
\varphi(S^{-1}\cjt_i ) = \left< z_{11}, z_{12}, \ldots, z_{1i} \right>
\]
is prime, so $S^{-1}\cjt_i$ is prime, as desired.

Now assume $1\leq i <j \leq n-1$. The ideal $\cj_i^t+\ck_j$ is saturated by Proposition~\ref{prop.saturated}. By the same reasoning as above, to complete our proof that $\cj_i^t+\ck_j$ is prime, it suffices to show that $\varphi(S^{-1}(\cj_i^t+\ck_j)) = \varphi(S^{-1}\cj_i^t)+\varphi(S^{-1}\ck_j)$ is prime in $\C(t)[\mathbf{z}]$, where we have already shown that $\varphi(S^{-1}\cj_i^t)$ and $\varphi(S^{-1}\ck_j)$ are prime. Since these latter ideals are generated by functions in disjoint sets of variables, their sum is prime. 
\end{proof}

\begin{thm}\label{thm.radical.prod}
The ideal $\cj_i^t \cdot \ck_j \subset \C[t,\mathbf{z}]$ is radical for all $i, j\in [n]$ with $i\leq j$. In particular, $\cjt_i \cdot \ck_j = \cjt_i\cap \ck_j$. 
\end{thm}

\begin{proof} The statement is trivially true when $j=n$ so we assume $j<n$. To prove the theorem, we prove that 
\begin{align}\label{eqn.G.set}
\{g_kp_B \mid  1 \leq k \leq i, & \,B\subseteq \{2, \ldots, n \}, |B| = j\}\\
\nonumber & \cup \{h_{k\ell} p_B\mid 1 \leq k<\ell \leq i, \,B\subseteq \{2, \ldots, n \}, |B| = j\}
\end{align}
is a Gr\"obner basis for $\cj_i^t \cdot \ck_j$ with respect to the lexicographic monomial order determined by the variable order~\eqref{eqn.order}.  Once this claim has been established, Corollary~\ref{cor.sq.free} implies that $\LT_{\preceq}(\cj_i^t\cdot \ck_j)$ is generated by square-free monomials. Recall an ideal whose initial monomial ideal is square-free is radical, so $\cj_i^t\cdot \ck_j$ is radical.

To prove the claim, we apply Theorem~\ref{thm.Grobner}.  We must compute all possible pairs of $S$-polynomials, each of which is of the form $S(fp,gq)$ for one of the following cases.
\begin{center}
\begin{tabular}{c|c|c|c|c|c}
 & $f$ & $g$ & $p$ & $q$ & index set \\ \hline
\multirow{2}{*}{Case 1}  & \multirow{2}{*}{$g_k$} & \multirow{2}{*}{$g_\ell$} & \multirow{2}{*}{$p_B$} &  \multirow{2}{*}{$p_C$} & $B, C \subseteq \{2, \ldots, n \},\; |B|=|C|=j$\\
& & & & & $1\leq k\leq i,\; 1\leq \ell \leq i$ \\ \hline
\multirow{2}{*}{Case 2} & \multirow{2}{*}{$g_k$} & \multirow{2}{*}{$h_{ab}$} & \multirow{2}{*}{$p_B$} & \multirow{2}{*}{$p_C$} & $B, C \subseteq \{2, \ldots, n \},\; |B|=|C|=j$\\
& & & & & $1\leq k\leq i,\; 1\leq a<b \leq i$ \\ \hline
\multirow{2}{*}{Case 3} & \multirow{2}{*}{$h_{k\ell}$} & \multirow{2}{*}{$h_{ab}$} & \multirow{2}{*}{$p_B$} & \multirow{2}{*}{$p_C$} & $B, C \subseteq \{2, \ldots, n \},\; |B|=|C|=j$\\
& & & & &  $1\leq k<\ell \leq i,\; 1\leq a<b \leq i$ \\ 
\end{tabular}
\end{center}
Let $G_i$ and $G^j$ be as defined in Corollary~\ref{cor.Grobner}, so that
\begin{align*}
    G_i &= \{g_k \mid 1\leq k\leq i\} 
    \cup \{h_{k\ell}\mid 1\leq k< \ell \leq i]\}, \mbox{and}  \\
    G^j &= \{p_B\mid B\subseteq \{2, \dots, n\}, |B|=j\}.
\end{align*} 
In each of the three cases above, $S(p,q)$ has a standard representation with respect to the set $G^j$ and $S(f,g)$ has a standard representation with respect to the set $G_i$  by Corollary~\ref{cor.Grobner}.  
Thus, in each case, Corollary~\ref{cor.sq.free} and Lemma~\ref{lemma.Spoly} together imply that $S(fp,gq)$ has a standard representation with respect to the set $G_iG^j$, given explicitly   in~\eqref{eqn.G.set}. Thus the set $G_iG^j$ is a Gr\"obner basis for $\cj_i^t\cdot \ck_j$. By Corollary~\ref{cor.sq.free}, its initial monomial ideal is square-free, implying that $\cj_i^t\cdot \ck_j$ is radical.

Applying Theorem~\ref{thm.prime} with $j=n$ implies that $\cj_i^t$ is prime and therefore radical.  Similarly, applying Theorem~\ref{thm.prime} with $i=0$ implies that $\ck_j$ is prime and therefore radical.  Since $\cj_i^t\cdot \ck_j$ is radical, we conclude
\[
\cjt_i \cdot \ck_j = \mathrm{rad}(\cjt_i\cdot \ck_j) = \mathrm{rad}(\cjt_i\cap  \ck_j) = \mathrm{rad}(\cjt_i) \cap \mathrm{rad}(\ck_j) = \cjt_i \cap \ck_j,
\]
as desired.
\end{proof}


\section{Flat family over the minimal sheet}\label{sec.FlatMinSheet}
In this section, we prove our main result in Theorem~\ref{thm.flat} by studying the ideal $\tilde{\ci}_{\mx_t, h}$ defining the family $\fX_{\mx_t,h}$ over the minimal sheet.  First, the following lemma tells us that it suffices to consider $\semi = \diag(1,0,\dots,0)$ and the sheet line $\semi_t$ as in~\eqref{eqn.xt.min}.

\begin{lemma}\label{lem.reduction} Up to isomorphism, there are only two distinct families $\fX_{\mx_t, h}$ of matrix Hessenberg schemes over the minimal sheet, namely $\fX_{\semi_t, h}$ and $\fX_{\nilp_t, h}$ where $\semi= \diag(1,0,\dots,0)$ and $\nilp = E_{n-1,n}$. Furthermore, the family $\varphi_{\nilp, h}:\fX_{\nilp_t, h}\to \C[t]$ is trivially flat.  
\end{lemma}
\begin{proof}  In light of Remark~\ref{rem.minimal}, it suffices to consider two cases. First, assume $\mx=\diag(c_1, c_2 \dots , c_2)$ with $c_1\neq c_2$. In this case, we have 
\[
\mx_t = t\diag(c_1,c_2,\dots,c_2)+E_{1n}.
\]
The rank conditions of \eqref{eqn.rank.cond} are satisfied for $\mx_t$ if and only if they are satisfied for 
$$\mx_t'=t\diag(c_1-c_2,0, \dots, 0)+E_{1n} = t(c_1-c_2)\semi+E_{1n}.$$  
It follows from Remark~\ref{rem:sameideal} that, for any Hessenberg function $h$, we have an equality $\ci_{\mx'_t,h}=\ci_{\mx_t,h}$. Recall from  Example~\ref{ex.min.sheet.x} that $\semi_t = t\semi+E_{1n}$. 
Consider the ring isomorphism $\C[\mathbf{z}, t] \to \C[\mathbf{z}, t]$ defined by $z_{1j} \mapsto (c_1-c_2)^{-1}z_{1j}$ for all $1\leq j \leq n$ and $z_{ij}\mapsto z_{ij}$ for all $1\leq i \leq n$ and $2\leq j \leq n$.  Comparing the rank conditions of~\eqref{eqn.rank.cond} for $\mx'_t$ and $\semi_t$, it is obvious that this isomorphism maps the ideal $\ci_{\mx'_t,h}$ isomorphically onto $\ci_{\semi_t, h}$.  

Next, suppose $\mx = cI_n+E_{n-1, n}$ for some $c\in \C$. In this case we have $\mx_t = tcI_n + E_{n-1, n}$.  The rank inequalities of \eqref{eqn.rank.cond} are satisfied for $\mx_t$ if and only if they are satisfied for $\nilp_t=E_{n-1, n}$. It follows from Remark~\ref{rem:sameideal} that, for any $h$, we have an equality $\fX_{\mx_t, h} = \fX_{\nilp_t, h}$. As $\nilp_t=\nilp$ the family $\varphi_{\nilp, h}$ is trivial and therefore flat.
\end{proof}

\subsection{Associated primes of $\fX_{\semi_t, h}$ over the minimal sheet} \label{se:AssociatedPrimesMinSheet} 
We assume from now on that $\semi_t$ is the sheet line~\eqref{eqn.xt.min} in the minimal sheet $\fg_{(2, 1^{n-2})}$. Our next theorem computes the associated primes of $\ci_{\semi_t, h}$ and proves $\tilde{\ci}_{\semi_t, h} = \ci_{\semi_t, h}$. For a given Hessenberg function $h$, define 
\begin{eqnarray}\label{eqn.prime.def}
\cpt_{i,h} := \cjt_{i-1}+\ck_{h(i)}. 
\end{eqnarray}
Since $i-1<i\leq h(i)$ for all $i$, Theorem~\ref{thm.prime} implies $\cpt_{i,h}$ is a prime ideal.  

Let 
\[
\cp_{i,h}^a:= \ev_a(\cpt_{i,h}) = \cj_{i-1}^a + \ck_{h(i)}, \; a\in \C
\] 
be the image of the prime ideals defined in~\eqref{eqn.prime.def} under the evaluation homomorphism. By Lemma~\ref{lem.ev2}, the ideal $\cp_{i,h}^a$ is prime whenever $a\neq 0$. More specifically, Lemma~\ref{lem.ev2}(2) tells us $\Spec\ \!\C[\mathbf{z}]/\cp_{i,h}^a$ is isomorphic to the matrix Richardson variety $\cx_{u[i+1]}^{\op} \cap \cx_{w_0v[j]}$. The authors proved in~\cite{GP-semisimple} that each irreducible component of $\fX_{\semi, h}$ is of this form and identify these components precisely in terms of the corners of $h$.  The theorem below generalizes that result from the semisimple matrix Hessenberg scheme $\fX_{\semi, h}$ to the family $\fX_{\semi_t, h}$. In the setting of the family considered here, the hard work has already been carried out in Section~\ref{sec.technical}. With the results of Theorem~\ref{thm.prime} and Theorem~\ref{thm.radical.prod} in place, the proof of Theorem~\ref{thm.assprimes} proceeds in a similar manner as that of~\cite[Thm.~4.1]{GP-semisimple}.

\begin{thm}\label{thm.assprimes}  Suppose $\semi_t$ is the line of matrices in the minimal sheet $\fg_{(2,1^{n-2})}$, defined as in~\eqref{eqn.xt.min}. For any Hessenberg function $h:[n]\to[n]$, 
\begin{equation}\label{eqn.intersection}
\ci_{\semi_t, h}  =  \bigcap_{i\in \cc(h)} \cpt_{i,h}.
\end{equation}
In particular, $\tilde{\ci}_{\semi_t, h} = \ci_{\mx_t, h}$, the affine scheme $\fX_{\semi_t,h}$ is reduced, and $\left\{\cpt_{i,h} \mid i\in \cc(h)\right\}$ is the set of associated primes of $\tilde{\ci}_{\semi_t, h}$.
\end{thm}

Our proof of Theorem~\ref{thm.assprimes} uses the following well known fact.

\begin{lemma}\label{lem.ideal.ops} For all ideals $\ci$, $\cj$, and $\ck$ in a commutative ring $R$, we have $(\ci\cap \cj)+(\ci\cap \ck) \subseteq \ci\cap (\cj+\ck)$.  Equality holds whenever either $\cj \subseteq \ci$ or $\ck\subseteq \ci$.
\end{lemma}

We are now ready to prove the theorem. 

\begin{proof}[Proof of Theorem~\ref{thm.assprimes}] Applying Lemma~\ref{lemma.generators} and Theorem~\ref{thm.radical.prod} we have
\begin{eqnarray}\label{eqn.decomp}
\ci_{\semi_t, h}  = \sum_{i\in \cc(h)} \left( \cjt_{i^*} \cap \ck_{h(i)} \right).
\end{eqnarray}
Proceed by induction on the number of corners of $h$. If $h$ has one corner then $h=(n,n,\ldots,n)$ and the assertion holds trivially.

Now suppose that $h$ has at least two corners.  Let $m=\max\{i\in \cc(h) \mid h(i)<n\}$, so the largest corner of $h$ is $m^*+1$.  Define $h':[n]\to [n]$ to be the Hessenberg function such that $h'(i)=h(i)$ for all $i<m$ and $h'(i)=n$ for all $i\geq m$.  Then $\cc(h') = \cc(h) \setminus \{m^*+1\}$ and $m$ is the largest corner of $h'$.  Notice that $h'(m)=n$, while $h(m)<n$ by definition.  This means that the corner $m$ contributes a term to the RHS of~\eqref{eqn.decomp} for $\ci_{\semi_t, h}$, but not for $\ci_{\semi_t, h'}$.  Since $h'(i)=h(i)$ for all $i<m$,
\begin{eqnarray}
\label{eqn.ci}
\sum_{i\in \cc(h)} \left( \cjt_{i^*} \cap \ck_{h(i)} \right) = \left( \sum_{i\in \cc(h')}  \left( \cjt_{i^*} \cap \ck_{h'(i)} \right) \right) + \left( \cjt_{m^*} \cap \ck_{h(m)} \right).
\end{eqnarray}
To simplify notation for the remainder of the proof, we write 
\begin{eqnarray}\label{eqn.cldef}
\cl = \bigcap_{\substack{i\in \cc(h)\\i<m}}\cpt_{i, h}. 
\end{eqnarray}
Note that as $h(i)\leq h(m)$ for all $i<m$ we have $\ck_{h(m)} \subseteq \cjt_{i-1}+\ck_{h(i)} = \cpt_{i,h}$ and we conclude that $\ck_{h(m)} \subseteq \cl$.

Applying the induction hypothesis and using the fact that $\cpt_{i,h'} = \cpt_{i, h}$ for all $i<m$ and that $\cpt_{m, h'} = \cjt_{m-1}+\ck_{h'(m)} = \cjt_{m-1}$ we have,
\begin{eqnarray}\label{eqn.IH}
 \sum_{i\in \cc(h')}  \left( \cjt_{i^*} \cap \ck_{h'(i)} \right) = \bigcap_{i\in \cc(h')} \cpt_{i, h'} = \left( \bigcap_{\substack{i\in \cc(h)\\i<m}}\cpt_{i, h} \right) \cap  \cjt_{m-1} = \cl \cap \cjt_{m-1}.
\end{eqnarray}
Combining equations~\eqref{eqn.ci} and~\eqref{eqn.IH} we have
\begin{eqnarray*}
\sum_{i\in \cc(h)} \left( \cjt_{i^*} \cap \ck_{h(i)} \right)  &=& \left( \cl \cap  \cjt_{m-1} \right) + (\cjt_{m^*}\cap \ck_{h(m)}) \\
&=& \left( \cl \cap  \cjt_{m-1} \cap \cjt_{m^*} \right) + (\cjt_{m^*}\cap \ck_{h(m)}) \\
&&\quad\quad\quad\quad\quad\textup{ since $\cjt_{m-1}\subseteq \cjt_{m^*}$ as $m-1<m\leq m^*$}\\
&=& \cjt_{m^*} \cap \left[ \left( \cl \cap  \cjt_{m-1} \right) +\ck_{h(m)} \right]\\
&&\quad\quad\quad\quad\quad \textup{ by Lemma~\ref{lem.ideal.ops}, as $\cl\cap \cjt_{m-1}\subseteq \cjt_{m^*}$} \\
&=&\cjt_{m^*} \cap \left[ \left( \cl \cap  \cjt_{m-1} \right) + (\cl \cap \ck_{h(m)}) \right] \quad  \textup{ since $\ck_{h(m)} \subseteq \cl$}\\
&=& \cjt_{m^*} \cap \left[ \cl \cap \left( \cj_{m-1}^t + \ck_{h(m)} \right) \right]  \quad \textup{ by Lemma~\ref{lem.ideal.ops}, as $\ck_{h(m)} \subseteq \cl$} \\
&=& \cl \cap  \left( \cj_{m-1}^t + \ck_{h(m)} \right)  \cap \cjt_{m^*} \\ 
&=& \bigcap_{i\in \cc(h)} \cpt_{i,h}
\end{eqnarray*}
where the last equality follows from~\eqref{eqn.cldef} and $\cpt_{m,h} = \cj_{m-1}^t + \ck_{h(m)} $ and $\cpt_{m^*+1, h} =  \cjt_{m^*}$. This proves~\eqref{eqn.intersection}. 

Each $\cpt_{i,h} = \cjt_{i-1}+\cjt_{h(i)}$ is prime by Theorem~\ref{thm.prime}, so~\eqref{eqn.intersection} implies that $\ci_{\mx_t, h}$ is an intersection of prime ideals and therefore radical.  By Lemma~\ref{lemma.inclusions}, $\cpt_{i,h}\subseteq \cpt_{j,h}$ if and only if $i\leq j$ and $h(i)\geq h(j)$.  This can never occur for distinct corners $i,j\in \cc(h)$  so the primes in the set $\left\{\cpt_{i,h} \mid i\in \cc(h)\right\}$ are distinct and none is contained in any of the others.  Thus~\eqref{eqn.intersection} implies that this is the set of associated primes of $\ci_{\semi_t, h}$. 

Recall that $d\in \C[\mathbf{z}]$ denotes the determinant function of a generic $n\times n$ matrix $Z=(z_{ij})_{1\leq i,j \leq n}$.  We have $\ev_a(d)=d$ where $\ev_a$ is the evaluation homomorphism of~\eqref{eqn.evaluationhom}.  Furthermore, for all $a\in \C^*$  we have that $\psi_a(d)$ is a constant multiple of $d$, where $\psi_a$ is the ring isomorphism from the statement of Lemma~\ref{lem.ev2}.  Together with Lemma~\ref{lem.ev2}(2), these observations imply that if $d\in \cpt_{i,h}$ then $d\in w_0\cdot\cj_{i-1}^0 + K_{h(i)}$, contradicting Lemma~\ref{lem.Schubert4}.  This proves $d\notin \cpt_{i,h}$ for all $i$ and thus $\tilde{\cp}^t_{i,h} = \cp^t_{i,h}$.  Since the localization map $\iota$ of equation~\eqref{eqn.GLn.coord} is injective it preserves intersections of ideals, and the previous observation implies $\tilde{\ci}_{\semi_t, h} = \ci_{\semi_t, h}$. The remaining statements for $\tilde{\ci}_{\semi_t,h}$ now follow from those proved for $\ci_{\semi_t, h}$ above.
\end{proof}

Our proof that $\varphi_{\mx, h}$ is flat follows immediately.

\begin{proof}[Proof of Theorem~\ref{thm.flat}] By Lemma~\ref{lem.reduction} it suffices to consider the family $\varphi_{\semi, h}:\fX_{\semi_t, h}\to \Spec\ \!\C[t]$ where $\semi = \diag (1,0,\dots, 0)$. According to Lemma~\ref{lemma.torsion}, we have only to prove that, for any linear polynomial $t+a\in \C[t]$ and any $f\in \C[t, \mathbf{z}]$, if $(t+a) f \in \ci_{\semi_t, h}$ then $f \in \ci_{\semi_t, h}$.  Since the family is reduced, its defining ideal is the intersection of its associated prime ideals. Thus by Theorem~\ref{thm.assprimes}, $(t+a)f \in \cpt_{i,h}$ for all $i\in \cc(h)$.  As each of ideals $ \cpt_{i,h}$ is prime and $t+a \notin \cpt_{i,h}$, we have $f \in  \cpt_{i,h}$ for all $i\in \cc(h)$.  Thus $f \in \ci_{\semi_t, h}$, as desired.
\end{proof}

Since our family of matrix Hessenberg schemes of $\fg_{(2,1^{n-2})}$ is defined by the ``intuitive ideal'' the same is true for each fiber.

\begin{cor}\label{cor.fiber.ideal} For all $\mx\in \fg_{(2,1^{n-2})}$ and all Hessenberg functions $h: [n]\to [n]$, we have $\tilde{\ci}_{\mx, h} = \ci_{\mx,h}$.  
\end{cor}
\begin{proof} By Lemma~\ref{lem:equality}, it suffices to assume $\mx = \semi = \diag (1,0,\dots, 0)$ or $\mx = \nilp = E_{1n}$.   Suppose $\semi_t$ is the line of matrices in the minimal sheet $\fg_{(2,1^{n-2})}$, defined as in~\eqref{eqn.xt.min}. By Lemma~\ref{lemma.fibers} the scheme-theoretic fiber of $\varphi_{\semi, h}$ is $\Spec\ \!\C[\mathbf{z}] / \tilde{\ci}_{\semi_a, h}$ and $\tilde{\ci}_{\semi_a, h} = \ev_a(\tilde{\ci}_{\semi_t, h})$. By Theorem~\ref{thm.assprimes}, we get $\tilde{\ci}_{\semi_a, h} = \ev_a(\ci_{\semi_t, h}) = \ci_{\semi_a, h}$ for all $a\in \C$.  Since $\semi_a$ and  $a\semi$ are conjugate matrices the previous conclusion implies $\tilde{\ci}_{a\semi} = \ci_{a\semi}$. Finally, it is straightforward to check that $\ci_{a\semi, h} = \ci_{\semi, h}$ for all $a\neq0$ (see Remark~\ref{rem:sameideal}), and thus $\tilde{\ci}_{\semi, h} = \ci_{\semi, h}$ as desired. Since $\semi_0=\nilp$, we also have $\tilde{\ci}_{\nilp ,h} = \ci_{\nilp ,h}$. 
\end{proof}

Recall from Section~\ref{sec.flat} that flat degeneration preserves the (multigraded) Hilbert series, and thus both the dimension and the cohomology of the matrix Hessenberg schemes over the minimal sheet remain constant across fibers; see Lemma~\ref{lem.flat} and Remark~\ref{rem.equalcohomologyclasses}.
We use Theorem~\ref{thm.flat} to compute these invariants of matrix Hessenberg schemes over the minimal sheet. The first result follows immediately from~\cite[Cor.~4.5]{GP-semisimple} and the second from~\cite[Cor.~4.6]{GP-semisimple}.

\begin{cor} \label{cor.dimension} For all $\mx\in \fg_{(2, 1^{n-2})}$,  
\[
\dim \fX_{\mx, h} = \frac{n(n+1)}{2} + \frac{(n-1)(n-2)}{2} + \max_{i\in \cc(h)} \{ h(i)-i \}.
\]
\end{cor}

\begin{cor} \label{cor.cohomclass}  Suppose $n\geq 3$, let $h$ be a Hessenberg function and $d_h= \max_i \{h(i)-i\}$.  For all $\mx\in \fg_{(2, 1^{n-2})}$, 
the cohomology class of the Hessenberg variety $Y_{\mx, h}$ is 
$$
[Y_{\mx, h}]= 
\begin{cases}
\displaystyle{\ \sum_{i=1}^{n-1}\ 2\ \mathfrak{S}_{w[i+1,i]}} & \mbox{if $h=(1,2,\dots, n)$}\\
\\
 \displaystyle{\sum_{\substack{i\in \cc(h)\ \mbox{\tiny{s.t.}}\\ h(i)-i=d_h}} \mathfrak{S}_{w[ i, h(i)]} }& \mbox{if $h\neq (1,2,\dots, n)$.}
\end{cases}
$$
where $w[i,j]$ is the shortest permutation $w$ in $S_n$ with $w(i)=1$ and $w(j) = n$.
\end{cor}

\begin{rem} The algebraic torus $T\simeq (\C^*)^n$ of $n\times n$ invertible matrices in $GL_n(\C)$ acts on both $Y_{\semi, h}$ and $Y_{\nilp, h}$ where $\semi=\diag(1,0,\dots, 0)$ 
and $\nilp=E_{1n}$.  However, we do not expect an variation of Corollary~\ref{cor.cohomclass} to hold in the setting of equivariant cohomology as the family $\fX_{\mx_t, h}$ is not $T$-equivariant. This is also obvious geometrically, as it is easy to see the Hessenberg varieties $Y_{\semi, h}$ and $Y_{\nilp, h}$ have different $T$-fixed points.  
\end{rem}

\subsection{Non-reduced Examples}\label{sec.examples}
Although $\fX_{\nilp, h}$ and $\fX_{\semi, h}$ lie in the same flat family, they may have different structure as schemes.  
One might hope that the results of Theorem~\ref{thm.assprimes} hold when passing to the fibers, that is, that the ideals $\cp_{i,h}^a$ are prime and taking those primes indexed by corners of $h$ yield the set of associated primes of $\ci_{\mx_a, h}$.  This intuition does indeed hold for $a\neq 0$, as discussed in the paragraph preceding Theorem~\ref{thm.assprimes} above. However, one does not obtain a comparable result for the special fiber $\fX_{\nilp, h}$ in all cases.

If some variation of Theorem~\ref{thm.assprimes} holds for all fibers, then all fibers of the family have components indexed by the corners of $h$ and therefore all fibers have the same number of irreducible components.  Our examples below show that this is not always the case.
In particular, Theorem~\ref{thm.radical.prod} was required in the proof of Theorem~\ref{thm.assprimes} to obtain the equality~\eqref{eqn.decomp}.  Note that homomorphisms do not respect intersections. That means that we cannot simply apply the evaluation homomorphism $\ev_a$ to obtain a variation of Theorem~\ref{thm.radical.prod} in the case $t=a$.  It might be the case that $\ev_a(\cjt_i\cdot\ck_j) = \cj_i^a \cdot \ck_j$ is not reduced, and when this occurs, Theorem~\ref{thm.assprimes} will not pass to the fiber over $a$.

\begin{example}\label{ex.1234}  A nilpotent matrix Hessenberg scheme may have fewer components than the associated semisimple matrix Hessenberg in the same sheet line of $\fg_{(2,1^{n-2})}$. Let $h=(1,2,3,4)$ and $\nilp=E_{14}$. We show here that $\fX_{\nilp, h}$ is a non-reduced nilpotent matrix Hessenberg scheme with no embedded primes, and with three components. Note that variety $X_{\nilp, h} = \pi(GL_n(\C)\cap \fX_{\nilp, h})$ is the Springer fiber over the minimal nilpotent orbit.

There are four corners for $h$, at $i=1, 2, 3, 4$ and $i^*=i$ for each corner. Three corners contribute to the defining ideal:
$$ 
\ci_{\nilp, h} = \ci_{\nilp, h, 1}+ \ci_{\nilp, h, 2} + \ci_{\nilp, h, 3}.
$$

Direct calculation shows that:
\begin{align*}
 \ci_{\nilp, h, 1}&=\langle z_{41}\rangle \cdot \langle z_{21}, z_{31}, z_{41}\rangle =  \cj_1^0 \cdot \ck_{1}\\
 \ci_{\nilp, h, 2}&=\langle z_{41}, z_{42}\rangle  \cdot \langle p_{\{2,3\}}, p_{\{2,4\}}, p_{\{3,4\}}\rangle = \cj_2^0\cdot \ck_2\\
  \ci_{\nilp, h, 3}&=\langle z_{41}, z_{42}, z_{43}\rangle  \cdot \langle p_{\{2,3, 4\}}\rangle = \cj_3^0 \cdot \ck_3.
  \end{align*}
 It is immediate that $\ci_{\nilp,h}$ is not radical, as $z_{41}^2\in \ci_{\nilp, h}$ but $z_{41}\not \in \ci_{\nilp, h}$ since the ideal is generated in degree 2 and higher.  
 By Corollary~\ref{cor.fiber.ideal}, $\tilde\ci_{\nilp,h}=\ci_{\nilp,h}$, so $\tilde{\ci}_{\nilp, h}$ is not radical and thus $\fX_{\nilp,h}$ is not reduced.
 
By Macaulay2, the associated primes  of  $\ci_{\nilp, h}$ and corresponding components are
\begin{align*}
\cj_3^0&=\langle z_{41}, z_{42}, z_{43}  \rangle, &\quad &\cx_{w_0[2341]}=\cx_{[3214]}\\
\cj_2^0 + \ck_2&= \langle z_{42}, z_{41}, z_{21}z_{32}-z_{22}z_{31}  \rangle,&\quad &\cx_{w_0[2413]}=\cx_{[3142]} \\
\ck_1&= \langle z_{21}, z_{31}, z_{41}  \rangle &\quad &\cx_{w_0[4123]}=\cx_{[1432]}.
\end{align*}
Here we have identified each associated prime as a Schubert determinantal ideal; note that the second of these  ideals is not considered in Lemma~\ref{lem.Schubert},
but follows from direct computation.
There are no embedded components, as the permutations $[3214]$, $[3142]$ and $[1432]$ are pairwise mutually incomparable with respect to the Bruhat order.  Thus $\fX_{\nilp, h}$ has 3 irreducible components. In contrast, the semisimple matrix Hessenberg scheme $\fX_{\semi, h=(1,2,3,4)}$
 is a reduced union of four matrix Richardson varieties (see \cite[Ex.~4.3]{GP-semisimple}).

By Corollary~\ref{cor.cohomclass}, the cohomology class of the Springer fiber $Y_{\nilp, h}$ is
$$
[Y_{\nilp, h}] = 2 \mathfrak{S}_{[1423]}+ 2\mathfrak{S}_{[2143]}+ 2\mathfrak{S}_{[2314]}.
$$
In particular, this implies that the multiplicity of each of the associated primes of $\ci_{\nilp ,h}$ is equal to $2$. A similar result holds for all $n$ when $h=(1,2,\ldots, n)$.
\end{example}

\begin{example} \label{ex.2244}  A nilpotent matrix Hessenberg scheme may have embedded components. Let $h=(2,2,4,4)$ and $\nilp=E_{14}$. In this case $\cc(h)=\{1, 3\}$ with $1^*=2$ and $3^*=4$. Only one corner contributes to the ideal $\ci_{\nilp, h}$ and a direct computation shows that
\[
\ci_{\nilp, h} =  \langle z_{41}, z_{42}\rangle  \cdot \langle p_{\{2,3\}}, p_{\{2,4\}}, p_{\{3,4\}}\rangle = \cj_2^0 \cdot \ck_2.
\]
Using Macaulay2, we see that the set of associated primes of $\ci_{\nilp, h}$ and corresponding components are
\begin{align*}
\cj_2^0&=\langle z_{41}, z_{42}  \rangle, &\quad & \cx_{w_0[2314]}=\cx_{[3241]}\\
\cj_2^0 + \ck_2 &= \langle z_{41}, z_{42}, z_{21}z_{32}-z_{22}z_{31}  \rangle,&\quad &\cx_{w_0[2413]}=\cx_{[3142]} \\
\ck_2&= \langle p_{\{2,3\}}, p_{\{2,4\}},p_{\{3,4\}}  \rangle &\quad &\cx_{w_0[1423]}=\cx_{[4132]}.
\end{align*}
Here we have identified each associated prime as a Schubert determinantal ideal; note the second of these  ideals is not considered in Lemma~\ref{lem.Schubert}, but follow from direct computation as in Example~\ref{ex.1234}.
We also observe that the ideal $\cj_2^0 + \ck_2$ corresponds to an embedded component, as $[3142] \leq_{\Br} [3241]$ in Bruhat order (note that $[3142] \leq_{\Br} [4132]$ also).  On the other hand, the permutations $[3214]$ and $[4132]$ are mutually incomparable with respect to Bruhat order.  Thus $\fX_{\nilp, h}$ is a non-reduced scheme with 2 irreducible components. In contrast, the semisimple matrix Hessenberg scheme $\fX_{\semi, h}$ is reduced and a union of two matrix Richardson varieties (see \cite[Ex.~4.3]{GP-semisimple}).
\end{example}

In contrast to these examples, recall from Example~\ref{ex.2444.2} that the semisimple and nilpotent Hessenberg schemes over the minimal sheet for $h=(2,4,4,4)$ are reduced, and each is is a union of two irreducible components (although neither is equidimensional).


\section{Nilpotent Matrix Hessenberg schemes over the minimal sheet} \label{sec.nilpotent}

This section explores the structure of nilpotent matrix Hessenberg schemes over the minimal sheet.   When $h$ is indecomposable, we identify all irreducible components of $\fX_{\nilp, h}$ and prove that each irreducible component of the semisimple matrix Hessenberg scheme $\fX_{\semi, h}$ admits a flat degeneration to an irreducible component of the nilpotent matrix Hessenberg scheme $\fX_{\nilp, h}$. Our second main theorem proves that $\fX_{\nilp, h}$ is reduced if and only if $h$ is indecomposable.

Suppose the Hessenberg function $h: [n]\to [n]$ is indecomposable, that is, suppose $h(i)>i$ for all $1\leq i <n$. 
In this case, we obtain an explicit description of the ideal $\ci_{\nilp, h}$, analogous to that of Theorem~\ref{thm.assprimes}. 

\begin{lemma}\label{lem.ev0.rad} For all $i,j\in [n]$ such that $i<j$, we have that 
\[
\cj_i^0 \cdot \ck_j = \cj_i^0 \cap \ck_j.
\]
\end{lemma}
\begin{proof} Let $\preceq$ be any diagonal term order on $\C[\mathbf{z}]$.  Then 
\[
\{p_B \mid B\subseteq \{2,\ldots, n\}, |B|=j\}
\]
is a Gr\"obner basis for $\ck_j$ and 
\[
\{ z_{n1}, z_{n2}, \ldots, z_{ni} \}
\]
is a Gr\"obner basis for $\cj_i^0$. We have that $z_{nk}$ divides the leading monomial $z_{b_1 1}z_{b_22} \cdots z_{b_j j}$ of $p_{\{b_1<b_2< \cdots< b_j\}}$ only if $b_j=n$ and $k=j$.  However, $k\leq i <j$ by assumption.  Therefore the leading monomials of any $p_B$ and any $z_{nk}$ are relatively prime. By Lemma~\ref{lemma.Spoly}  and Theorem~\ref{thm.Grobner}, the set
\[
\{ z_{nk}p_B \mid B\subseteq \{2,\ldots, n\}, |B|=j, 1\leq k \leq i \}
\]
is a Gr\"obner basis for $\cj_i^0 \cdot \ck_j$. The leading term of every polynomial in this set is square free, so $\cj_i^0 \cdot \ck_j$ is radical and the desired equality follows immediately, as in the last sentence of the proof of Theorem~\ref{thm.radical.prod}.
\end{proof}

The following proposition is a scheme theoretic version of a result due to Tymoczko \cite[\S2.1]{Tymoczko2006A} describing the irreducible components of the nilpotent Hessenberg varieties $Y_{\nilp, h}$ in the flag variety $G/B$.

\begin{prop}\label{prop.nilp.components} Suppose $h: [n]\to [n]$ is an indecomposable Hessenberg function and $\nilp=E_{1n}$. Then 
\[
\ci_{\nilp, h} = \bigcap_{i\in \cc(h)} \left( \cj_{i-1}^0 + \ck_{h(i)} \right).
\] 
In particular, $\left\{ \cj_{i-1}^0 + \ck_{h(i)}  \mid i\in \cc(h) \right\}$ is the set of associated primes of ${\ci}_{\nilp, h}$.  Furthermore, the nilpotent matrix Hessenberg scheme $\fX_{\nilp, h}$ is reduced and a union of matrix Schubert varieties,
\[
\fX_{\nilp, h} = \bigcup_{i\in \cc(h)} \cx_{w_0w[i, h(i)]}
\]
where $w[i,h(i)] = u[i]v[h(i)]$ is the shortest permutation $w$ in $S_n$ such that $w(i)=1$ and $w(h(i))=n$.
\end{prop}

\begin{proof}  By Corollary~\ref{cor.fiber.ideal}, the matrix Hessenberg scheme $\fX_{\nilp, h}$ is defined by the ideal $\ci_{\nilp,h}$. Note that $\nilp = \semi_0$ on the sheet line $\semi_t$. Since the evaluation homomorphism $\ev_0$ preserves sums and products of ideals, Lemma~\ref{lemma.generators} implies
\[
\ci_{\nilp, h}  = \sum_{i\in \cc(h)} \left( \cj_{i^*}^0 \cdot \ck_{h(i)} \right).
\]
Since $h$ is indecomposable, we have $i^*<h(i^*)=h(i)$ so by Lemma~\ref{lem.ev0.rad},
\[
\ci_{\nilp, h}  = \sum_{i\in \cc(h)}  \left(  \cj_{i^*}^0 \cap \ck_{h(i)} \right).
\] 
The results of Lemma~\ref{lemma.inclusions} hold after replacing $t$ with $a$ since inclusions and sums are preserved under the evaluation homomorphism.  
The inductive portion of the proof of Theorem~\ref{thm.assprimes} gives a proof of the following equality, by replacing $t$ with $0$.
\[
\sum_{i\in \cc(h)} \left( \cj_{i^*}^0 \cap \ck_{h(i)} \right) =  \bigcap_{i\in \cc(h)}  \cp_{i,h}^0,
\]
where $\cp_{i,h}^0:=\cj_{i-1}^0 + \ck_{h(i)}$.  
It follows that
\[
\ci_{\nilp, h}  = \bigcap_{i\in \cc(h)} \cp_{i,h}^0 = \bigcap_{i\in \cc(h)} \left( \cj_{i-1}^0 + \ck_{h(i)} \right).
\]

By Lemma~\ref{lem.Schubert}, $\cj_{i-1}^0 + \ck_{h(i)}$ is prime, so $\ci_{\nilp, h}$ is an intersection of prime ideals and is therefore radical.  Furthermore, as in the proof of Theorem~\ref{thm.assprimes}, the set of prime ideals $\left\{ \cj_{i-1}^0 + \ck_{h(i)}  \mid i\in \cc(h) \right\}$ are distinct and none is contained in any of the others.  Thus, these are precisely the associated primes of $\ci_{\nilp, h}$ as claimed.  The geometric description of the scheme $\fX_{\nilp, h}$ as a union of matrix Schubert varieties now follows by Lemma~\ref{lem.Schubert}.
\end{proof}

\begin{thm}\label{thm.components.indecomp} Suppose $h$ is indecomposable, and let $\semi = \mathrm{diag}(1,0,\ldots, 1)$ and $\nilp = E_{1n}$. Each irreducible component of $\fX_{\semi, h}$ admits a flat degeneration to a unique irreducible component of $\fX_{\nilp, h}$, resulting in a bijective correspondence between the irreducible components of $\fX_{\semi, h}$ and that of $\fX_{\nilp,h}$.
\end{thm}
\begin{proof} For each $i\in \cc(h)$, consider the family of schemes defined by the morphism
\[
\varphi_{i,h}: \Spec\ \!\C[t, \mathbf{z}]/ \cpt_{i,h} \to \Spec\ \!\C[t]
\]
corresponding to the natural inclusion $\C[t] \hookrightarrow \C[t, \mathbf{z}]/ \cpt_{i,h}$, where $\cpt_{i,h}:= \cjt_{i-1} + \ck_{h(i)}$. By Theorem~\ref{thm.prime} 
the ideal $\cpt_{i,h}$
is prime.  Since $t-a\notin \cpt_{i,h}$ for all $a\in \C$, it follows from Lemma~\ref{lemma.torsion} that $\varphi_{i,h}$ is flat. The scheme theoretic fiber of $\varphi_{i,h}$ over $\left< t-a \right>$ is precisely $\Spec\ \!\C[\mathbf{z}] / \cp_{i,h}^a$.  By Lemma~\ref{lem.ev2}(2), the general fiber is the matrix Richardson variety $\cx_{u[i+1]}^{\op} \cap \cx_{w_0v[h(i)]}$ which is the irreducible component of $\fX_{\semi, h}$ indexed by corner $i$~\cite[Theorem 1]{GP-semisimple}.  By Proposition~\ref{prop.nilp.components} the special fiber is precisely the irreducible component $\cx_{w_0w[i,h(i)]}$ of $\fX_{\nilp,h}$ indexed by corner $i$. 
\end{proof}

The results of Theorem~\ref{thm.components.indecomp} need not hold when $h$ decomposable. Indeed, in Example~\ref{ex.1234} we observed that the nilpotent matrix Hessenberg scheme $\fX_{\nilp ,(1,2,3,4)}$ has three irreducible components, while $\fX_{\semi, (1,2,3,4)}$ has four irreducible components.

\begin{thm}\label{thm:decomposable} Let $h:[n] \to [n]$ be a  Hessenberg function. The matrix Hessenberg scheme $\fX_{\nilp, h}$ is reduced if and only if $h$ is indecomposable.
\end{thm}

\begin{proof}
By Proposition~\ref{prop.nilp.components}, we need only prove that $h$ decomposable implies $\fX_{\nilp, h}$ is not reduced. Suppose then that $h$ is decomposable.
 By Corollary~\ref{cor.fiber.ideal}, it suffices to exhibit a nonzero nilpotent element of $\C[\mathbf{z}]/\ci_{\nilp, h}$.  By Lemma~\ref{lemma.generators},
\begin{equation}\label{ideal.sum}
\ci_{\nilp, h} = \sum_{k\in \cc(h)} \ci_{\nilp, h,k^*} =  \sum_{k\in \cc(h)} \cj^0_{k^*}\cdot \ck_{h(k)}.
\end{equation}
For $h$ decomposable, the set $\cd(h):=\{j<n \mid h(j)=j\}$ is nonempty.  For each $j\in [n]$, let $j_*$ denote the minimal $k$ such that $h(k)=h(j)$. Recall that, for each $j\in [n]$, $j^*$ denote the maximal element $k$ such that $h(k)=h(j)$.  For each $j\in \cd(h)$, $j_*\in \cc(h)$ and $(j_*)^*=j$. In particular, for any $j\in \cd$, the index $j_*\in\cc(h)$ appears in the sum on the RHS of~\eqref{ideal.sum}, with corresponding summand $\ci_{\nilp, h, j}$. Furthermore, $h(j_*)=h(j)=j$ for all  $j\in \cd(h)$, so the sum~\eqref{ideal.sum} includes all terms  
\[
\ci_{\nilp,h,j}=\cj_j^0 \cdot \ck_j \;\; \mbox{for all } \;\; j\in \cd(h).
\]
Choose $i_0\in \cd$ minimal.
Let $B_0\subset \{2,\dots, n\}$ be a subset of size $i_0$ with $n\in B_0$. To complete the proof, we show that $p_{B_0}\not \in \ci_{\nilp, h}$ but $p_{B_0}^2\in \ci_{\nilp, h}$.

Observe $p_{B_0}\in \ck_{i_0}$ (as it's a generator) and $p_{B_0}\in\cj_{i_0}$ (by the choice of $B_0$, since $n\in B_0$), implying $p_{B_0}^2\in\cj_{i_0}\cdot \ck_{i_0}\subset \ci_{\nilp, h}$.

Suppose by way of contradiction that $p_{B_0}\in  \ci_{\nilp, h}$. All polynomials in $\ci_{\nilp,h,k^*}=\cj_{k^*}^0 \cdot \ck_{h(k)}$ have degree greater than or equal to $h(k)+1$.  If $k\geq (i_0)_*$, 
$$h(k)+1\geq h((i_0)_*)+1 = h(i_0)+1 = i_0+1,$$
and thus the elements of the ideals $\ci_{\nilp, h, k}$ from~\eqref{ideal.sum} with $k\geq (i_0)_*$ have degree larger than $i_0=\deg p_{B_0}$. Our assumption that $p_{B_0}\in \ci_{\nilp, h}$ now implies 
$$
p_{B_0}\in \sum_{\substack{k\in \cc(h)\\ k < (i_0)_*}} \ci_{\nilp,h,k^*},
$$ 
where the sum is nonempty since $p_{B_0}$ is nonzero. 

Since $i_0$ is minimal among elements of $\cd(h)$, for each $k<(i_0)^*$ we must have $k<h(k)$. This implies $ \ci_{n,h,k^*}= \cj_{k^*}^0\cdot \ck_{h(k)} =  \cj_{k^*}^0\cap \ck_{h(k)}$ by Lemma~\ref{lem.ev0.rad}. Thus
$$
p_{B_0}\in \sum_{\substack{k\in \cc(h)\\ k < (i_0)_*}} \cj_{k^*}^0\cap \ck_{h(k)}\subseteq \sum_{\substack{k\in \cc(h)\\ k < (i_0)_*}} \cj_{k^*}^0.
$$

Let $m=\max\{k\in \cc(h)\mid k < (i_0)_*\}$. Then $p_{B_0}\subset \cj_{m^*}^0$ by Lemma~\ref{lemma.inclusions} (setting $t=0$). 
Since $n\in {B_0}$, we may write $p_{B_0}$ using Laplace expansion across the bottom row as a sum
\begin{equation}\label{eq:pBsum}
\pm p_{B_0} = \sum_{s=1}^{m^*}(-1)^s z_{ns} q^{(s)} + \sum_{s={m^*+1}}^{i_0} (-1)^s z_{ns} q^{(s)} 
\end{equation}
for polynomials $q^{(s)}$, where the second sum is nontrivial as $i_0>m^*$. Specifically,
$q^{(s)}$ is the determinant of $({\bf z}_{\alpha\beta})$ where the indices run over $\alpha\in B_0\setminus \{n\}$ and $\beta\in \{1,\dots, i_0\},\; \beta\neq s$. 

Let 
$$
g_{\leq m^*} = \sum_{s=1}^{m^*}(-1)^s z_{ns} q^{(s)} \quad \mbox{ and}\quad
 g_{>m^*}=\sum_{s={m^*+1}}^{i_0} (-1)^s z_{ns} q^{(s)}.
$$
Then  $g_{\leq m^*} \in \cj_{m^*}^0=\langle z_{n1}, z_{n2}, \ldots , z_{nm^*}\rangle$ since each term in the sum is a multiple of $z_{ns}$ with $s \leq m^*$. Since we also have $p_{B_0}\in  \cj_{m^*}^0$,
\begin{equation}\label{eq:falseelement}
g_{>m^*}=\pm p_{B_0}- g_{\leq m^*}
\in  \cj_{m^*}^0 = \langle z_{n1}, z_{n2}, \ldots , z_{nm^*}\rangle.
\end{equation}
But $g_{>m^*}$ is a  polynomial in variables $z_{ij}$ for which $i=n$ implies $j >  m^*$, contradicting \eqref{eq:falseelement}.  We conclude that $p_{B_0}\not\in \ci_{\nilp, h}$. 
 
 Since $p_{B_0}$ is a nonzero nilpotent element of $\C[\mathbf{z}]/\ci_{\nilp, h}$, the scheme $\fX_{\nilp, h}$ is not reduced.\end{proof}


\appendix

\section{Miracle flatness revisited}\label{sec.miracleflatnessproofs}
This appendix explores applications of miracle flatness (Theorem~\ref{thm.miracle}) to the family $\fX_{\mx_t, h}\rightarrow \Spec\ \!\C[t]$. It may be used to show flatness in the following cases:
\begin{enumerate}  
\item for all indecomposable Hessenberg functions and matrices $\mx$ in the regular sheet $\fg_{(n)}$, and 
\item for all matrices $\mx\in \mathfrak{gl}_n(\C)$, 
when $h = (n-1, n,\ldots, n)$.
\end{enumerate}
Cases (1) and (2) have either appeared in the literature or are easily obtained from results in the literature. We reprove these results here in order to streamline the arguments and showcase how this method differs from that of our proof of Theorem~\ref{thm.flat}. 

To begin, we record a helpful lemma which follows from work of Insko, Tymoczko, and Woo; see~\cite[Section 4]{Insko-Tymoczko-Woo}. Note that the results in~\cite{Insko-Tymoczko-Woo} consider matrices with coefficients in a field, but each generalizes to the setting of matrices with coefficients in $\C[t]$ without any changes to the proofs.

\begin{lemma} \label{lemma.LCIgens} Let $R=\C$ or $R=\C[t]$. For any Hessenberg function $h: [n] \to [n]$ and matrix $\mx\in M_n(R)$, the ideal $\iota(\ci_{\mx_t, h})R[\mathbf{z}, d^{-1}]$ in $R[\mathbf{z}, d^{-1}]$ is generated by $n^2-\sum_{i=1}^n h(i)$ polynomials in $R[\mathbf{z}, d^{-1}]$.
\end{lemma}

We obtain two immediate corollaries.

\begin{cor} \label{cor.dim} For all $\mx\in \mathfrak{gl}_n(\C)$ and Hessenberg functions $h: [n]\to [n]$, we have  $\dim \,\fX_{\mx_t, h} \geq  1+\sum_{i=1}^n h(i)$.
\end{cor}
\begin{proof}  Observe that $\cx := \fX_{\mx_t, h} \cap \left( G \times \Spec\ \!\C[t] \right)$ is a dense open subset of $\fX_{\mx_t, h}$, so $ \dim \cx = \dim \fX_{\mx_t, h}$. On the other hand, $\cx =  \Spec \ \! \C[t,\mathbf{z}, d^{-1}] / \cm_t$ where $\cm_t:=\iota(\ci_{\mx_t, h})R[\mathbf{z}, d^{-1}]$. 
By Lemma~\ref{lemma.LCIgens}, $\cm_t$ is generated by $\sum_{i=1}^n (n-h(i))$ polynomials. Thus
\begin{align*}
\dim \fX_{\mx_t, h}&=\dim \cx 
= \dim G + 1 -\dim \cm_t\\ 
&\geq  n^2 +1 - \left(n^2- \sum_{i=1}^n h(i)\right)= 1+\sum_{i=1}^n h(i),
\end{align*}
as claimed.
\end{proof}

As an immediate corollary to the proof, we obtain the following.

\begin{cor} \label{cor.CM} If $\dim \, \fX_{\mx_t, h} = 1 + \sum_{i=1}^n h(i)$, the scheme 
$$
\mathcal X:= \fX_{\mx_t, h} \cap \left( G \times \Spec\ \!\C[t] \right)
$$ is a local complete intersection and therefore Cohen-Macaulay. 
\end{cor}
\begin{proof} Observe that $\cx$ is a closed subscheme of $G\times \Spec\ \!\C[t]$.
Let $\cm_t:= \iota(\ci_{\mx_t, h})R[\mathbf{z}, d^{-1}]$.
Since $\mathcal X = \Spec 
\ \C[t,\mathbf{z}, d^{-1}] / \cm_t$ is a dense open subset of $\fX_{\mx_t,h}$, its codimension in $G\times \Spec\ \!\C[t]$ is 
$$
n^2+1-\left(1+\sum_i h(i)\right) = n^2-\sum_i h(i).
$$
On the other hand, $\cm_t$ is generated by  $n^2-\sum_i h(i)$ polynomials, and thus $\mathcal X$ is a local complete intersection. Finally, every local complete intersection is Cohen--Macaulay~\cite[Prop.~8.23]{Hartshorne}.
\end{proof}

We apply miracle flatness (Theorem~\ref{thm.miracle}) to 
give us one possible criterion for proving Conjecture~\ref{conj}. The following is a restatement of Proposition~\ref{prop.miracle}.  

\begin{prop}\label{prop.flat.criterion2}
Given $\mx\in \mathfrak{gl}_n(\C)$ and a Hessenberg function $h:[n]\to[n]$, let $\varphi_{\mx, h}: \fX_{\mx_t, h}\to \Spec\ \!\C[t]$ be the morphism defined as in~\eqref{eqn.morphism}.  Suppose that for each $a\in \C$, the matrix Hessenberg scheme $\fX_{\mx_a,h}$ is irreducible and has dimension $\sum_{i=1}^n h(i)$.  Then $\varphi_{\mx, h}$ is flat.
\end{prop}
\begin{proof} 
By Lemma~\ref{lemma.fibers} the matrix Hessenberg scheme $\fX_{\mx_a, h}$ is precisely the scheme-theoretic fiber of $\varphi_{\mx, h}$ over $\left< t-a \right>\in \Spec\ \!\C[t]$. Since all fibers are irreducible and $\Spec\ \!\C[t]$ is irreducible, it follows immediately that the scheme $\fX_{\mx_t, h}$ is irreducible (see~\cite[1.\S6.3, Thm~1.26]{Shafarevich}).  Furthermore, its dimension is bounded above: 
\[
\dim \, \fX_{\mx_t} \leq \dim\, \fX_{\mx_a, h} +\dim \Spec\ \!\C[t]  = 1+\sum_{i=1}^n h(i);
\]
see \cite[Thm.~10.10]{Eisenbud} or~\cite[1.\S6.3, Thm~1.25]{Shafarevich}.
Combining this with the lower bound of Corollary~\ref{cor.dim} implies that equality must hold, and thus 
\[
\mathcal X := \fX_{\mx_t, h} \cap \left( G \times \Spec\ \!\C[t] \right) = \Spec\ \!\C[\mathbf{z}, t, d^{-1}]/ \cm_t,
\] 
where $\cm_t:=\iota(\ci_{\mx_t, h})R[\mathbf{z}, d^{-1}]$
is irreducible and Cohen--Macaulay by Corollary~\ref{cor.CM}.
We consider the morphism $\overline{\varphi}$ obtained by restricting $\varphi_{\mx, h}$ to $\mathcal X$,
\[
\bar{\varphi}:\mathcal X \to \Spec\ \!\C[t].
\]
The scheme theoretic fiber of $\bar{\varphi}$ over $\left<  t-a \right>\in \Spec\ \!\C[t]$ is
$$
\fX_{\mx_a, h} \cap G  = \Spec\ \!\C[\mathbf{z}, d^{-1}] /\cm_a.
$$  Thus, the assumption that $\dim \fX_{\mx_a ,h} = \sum_{i=1}^n h(i)$ implies the dimensions of all fibers of $\bar{\varphi}$ are equal and satisfy
\begin{eqnarray}\label{eqn.dim}
\dim \cx=
\dim \Spec\ \!\C[t]+ \dim \Spec\ \!\C[\mathbf{z}, d^{-1}] / \cm_a.
\end{eqnarray}
Theorem~\ref{thm.miracle} now implies $\bar{\varphi}$  is flat.

Finally, we show that 
$\bar{\varphi}$ flat implies that $\varphi_{\mx, h}$ is  flat.  Recall that $\tilde{\ci}_{\mx_t, h}$ is the saturation of $\ci_{\mx_t,h}$ with respect to the determinant function $d$ as in~\eqref{eq.definitionItilde}.  
By Lemma~\ref{lemma.torsion} it suffices to show that if $f\in \C[t, \mathbf{z}]$ such that $(t-a)f\in \tilde{\ci}_{\mx_t, h}$ for some $a\in \C$, then $f\in \tilde{\ci}_{\mx_t, h}$.   Indeed, 
\begin{eqnarray*}
(t-a)f\in \tilde{\ci}_{\mx_t, h} \Rightarrow (t-a)\iota(f) \in \iota (\tilde{\ci}_{\mx_t, h}) = \iota (\ci_{\mx_t, h}) \Rightarrow \iota(f)\in \iota (\ci_{\mx_t, h}) \Rightarrow f\in \tilde{\ci}_{\mx_t, h}
\end{eqnarray*}
where the second implication follows from Lemma~\ref{lemma.torsion} since $\bar{\varphi}$ is flat.  
\end{proof}

We now apply Proposition~\ref{prop.flat.criterion2} to conclude $\varphi_{\mx, h}$ is flat in the two cases introduced at the beginning of this section.

\begin{prop} Suppose $\mx\in \mathfrak{gl}_n(\C)$ and $h: [n]\to [n]$ is a Hessenberg function.  Assume further that either
\begin{enumerate}
\item $\mx$ is an element of the regular sheet $\fg_{reg}$ and $h(i)>i$ or that
\item $h = (n-1, n, \ldots, n)$.
\end{enumerate}
Then $\varphi_{\mx, h}$ is flat.
\end{prop}

\begin{proof} It suffices to argue that every Hessenberg variety $Y_{\mx, h}\subseteq GL_n(\C)/B$ in the cases above is irreducible of dimension $\sum_{i=1}^n (h(i)-i)$.  Indeed, given that information, the dimension of the matrix Hessenberg scheme $\fX_{\mx, h}$ is 
\[
\dim B + \sum_{i=1}^n(h(i)-i) = \frac{n(n+1)}{2} +\sum_{i=1}^n(h(i)-i)  = \sum_{i=1}^n h(i),
\]
and the fact that $\fX_{\mx, h}$ is the closure of $\pi^{-1}(Y_{\mx, h})$ in $\Spec\ \!\C[\mathbf{z}]$ implies $\fX_{\mx, h}$ is also irreducible.

If $\mx$ is a regular matrix and $h$ is indecomposable then the fact that $Y_{\mx, h}$ is irreducible and of the required dimension was proved by the second author in~\cite{Precup2018}. (The fact that $Y_{\mx, h}$ is irreducible when $\mx$ is regular nilpotent or regular semisimple appeared much earlier in the literature; see~\cite{Anderson-Tymoczko2010, DPS1992}.)

Next, we assume $h=(n-1, n,\ldots, n)$. If $\mx\in \fg_{(1^n)}$ then $\mx$ is either the zero matrix or a scalar multiple of the identity.  It's obvious that $\tilde{\ci}_{\mx, h} = \ci_{\mx, h}=\left< 0\right>$ in this case and $\varphi_{\mx, h}$ is a trivial vector bundle, and therefore flat.  Next, if $\mx\in \fg_{(2,1^{n-2})}$ then $\varphi_{\mx, h}$ is flat by the main theorem of this text, Theorem~\ref{thm.flat}.  We may therefore assume $\mx \notin \fg_{(1^n)}$ and $\mx\notin \fg_{(2,1^{n-2})}$.  The fact that $h=(n-1, 1, \ldots, 1)$ implies $\ci_{\mx, h}$ is principal for all $\mx\in \mathfrak{gl}_n(\C)$ and it follows immediately that $\dim Y_{\mx, h} = n^2-1 = \sum_{i=1}^n (h(i)-i)$. Finally, $Y_{\mx, h}$ is irreducible whenever $\mx \notin \fg_{(1^n)}$ and $\mx\notin \fg_{(2,1^{n-2})}$ by~\cite[Thm.~4]{EPS-online}.
\end{proof}

\newcommand{\etalchar}[1]{$^{#1}$}

\end{document}